\documentclass{article} 

\usepackage{amsmath,amssymb,amsthm}
\usepackage{enumitem}
\usepackage{fancyhdr}
\usepackage{xparse, xcolor, graphicx}
\usepackage{cite}
\usepackage{tikz}
\usetikzlibrary{calc}

\NewDocumentCommand{\code}{v}{\texttt{#1}}

\newtheorem{theorem}{Theorem}[section] 
\newtheorem{lemma}[theorem]{Lemma} 
\newtheorem{prop}[theorem]{Proposition}

\newtheorem{cor}[theorem]{Corollary}

\theoremstyle{definition}
\newtheorem{definition}[theorem]{Definition}
\newtheorem{defn}[theorem]{Definition}

\newtheorem{conjecture}[theorem]{Conjecture}
\newtheorem{fact}[theorem]{Fact}

\theoremstyle{remark}

\DeclareMathOperator{\Av}{Av}
\DeclareMathOperator{\cl}{cl}
\DeclareMathOperator{\id}{id}
\DeclareMathOperator{\Int}{Int}
\DeclareMathOperator{\LL}{{\Lambda\hspace{-.095cm}\Lambda}}
\DeclareMathOperator{\lng}{long}
\DeclareMathOperator{\nw}{nwr}

\DeclareMathOperator{\swl}{swl}
\DeclareMathOperator{\swr}{swr}
\DeclareMathOperator{\VHC}{VHC}

\DeclareMathOperator{\RedVHC}{RedVHC}
\newcommand{\E}{{\mathbb E}}
\renewcommand{\H}{{\mathcal H}}
\newcommand{\M}{{\mathcal M}}
\newcommand{\N}{{\mathcal N}}
\newcommand{\R}{{\mathfrak R}}
\newcommand{\s}{{\mathfrak s}}
\newcommand{\w}{{\mathfrak w}}
\newcommand{\W}{{\mathfrak W}}

\title{Further Bijections to Pattern-Avoiding Valid Hook Configurations}
\author{Maya Sankar}
\date{}
\makeatletter
\let\mytitle\@title
\let\myauthor\@author
\makeatother
\begin{document}

\nocite{*}

\maketitle

\begin{abstract}
Valid hook configurations are combinatorial objects used to understand West's stack-sorting map. We extend existing bijections corresponding valid hook configurations to intervals in partial orders on Motzkin paths. To enumerate valid hook configurations on $312$-avoiding permutations, we build off of an existing bijection into a Motzkin poset and construct a bijection to certain well-studied closed lattice walks in the first quadrant. We use existing results about these lattice paths to show that valid hook configurations on $312$-avoiding permutations are not counted by a $D$-finite generating function, resolving a question of Defant's, and additionally to compute asymptotics for the number of such configurations. We also extend a bijection of Defant's to a correspondence between valid hook configurations on $132$-avoiding permutations and intervals in the Motzkin-Tamari posets, providing a more elegant proof of Defant's enumeration thereof. To investigate this bijection, we present a number of lemmas about valid hook configurations that are generally applicable and further study the bijections of Defant's.
\end{abstract}

\section{Introduction}
In 1990, West defined what we now call the \emph{stack-sorting map}, which maps permutations to permutations that are closer to the ``sorted'' increasing permutation \cite{west}. It became well-studied because of its relation to Knuth's ``stack-sorting algorithm'' defined in \emph{The Art of Computer Programming} over twenty years prior, which had engendered significant progress in combinatorics and computer science \cite{knuth}. A question that has garnered a lot of interest is to determine the \emph{fertility} of a given permutation $\pi$, or the number of preimages $\pi$ has under the stack-sorting map. West himself expended a lot of effort to compute fertilities of a few specific classes of permutations, and ten years later Bousquet-M\'elou provided an algorithm to determine whether a permutation had nonzero fertility \cite{bm}. However, the task of explicitly determining the fertility of any permutation remained open. See \cite{Bona, BonaSurvey, bm2, bm, DefantFertilityWilf, DefantPolyurethane, Fang2, knuth, west, Zeilberger} for more information about the stack-sorting map

In 2017, Defant presented an approach to compute the fertility of an arbitrary permutation using auxiliary structures on permutations called valid hook configurations (defined in Section \ref{section:defs}) \cite{DefantPreimages}. This new approach further allowed Defant to generalize existing theorems about the stack-sorting map and prove new results \cite{DefantDescents, DefantFertility, DefantPreimages, DefantPostorder}, providing a new framework for us to understand the stack-sorting map. The relevance of valid hook configurations prompted mathematicians to study them as combinatorial objects in their own right. In 2018, Defant, Engen, and Miller showed that valid hook configurations of length-$n$ permutations are in bijective correspondence with certain weighted set partitions \cite{DefantEngenMiller} that Josuat-Verget studied in the context of free probability theory \cite{jv}. In 2019, Defant undertook to enumerate valid hook configurations of length-$n$ permutations avoiding one or two patterns of length 3 \cite{Defant}. It is with this question we primarily concern ourselves.

\subsection{Posets on Motzkin Paths}

Defant's results primarily involved bijections to Motzkin intervals. A \emph{Motzkin path} is a lattice path consisting of steps going \emph{up} by $(1,1)$, \emph{down} by $(-1,1)$, or \emph{east} by $(1,0)$ that starts and ends on the $x$-axis and never dips below it. Let $U$, $D$, and $E$ represent up, down, and east respectively. We will typically denote a Motzkin path $P$ of length $n$ as a sequence $P_1\cdots P_n$ of elements in $\{U,D,E\}$ such that there is the same number of $U$'s and $D$'s in the full sequence $P_1\cdots P_n$ and every prefix $P_1\cdots P_i$ has at least as many $U$'s as $D$'s.
\begin{figure}[h]
\begin{center}
\begin{tikzpicture}[scale=.5]
\draw[step=1, gray, thin] (0,0) grid (8,2);
\draw[thick] (0,0) -- ++ (1,1) -- ++ (1, -1) --++(1,1)--++(1,0)--++(1,1)--++(1,-1)--++(1,-1)--++(1,0);
\end{tikzpicture}
\caption{The Motzkin path $UDUEUDDE$.}
\end{center}  
\end{figure}
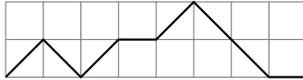

Let $M_n$ be the set of Motzkin paths of length $n$. There is a natural partial order $\leq_S$ on $M_n$ where $P\leq_SQ$ if $P$ lies below or is equal to $Q$ everywhere. Alternatively, we say that $P\leq_SQ$ if the prefix $P_1\cdots P_i$ contains at least as many $U$'s as the corresponding prefix $Q_1\cdots Q_i$ for each $i\in[n]$. Defant denotes this poset as $\M_n^S:=(M_n,\leq_S)$ and calls it the $n^\text{th}$ \emph{Motzkin-Stanley} lattice. This references the \emph{Stanley lattice} defined by Bernardi and Bonichon \cite{bb} which restricts $\M_n^S$ to paths without east steps; Ferrari and Pinzani proved that both partial orders are lattices \cite{Ferrari}. Defant additionally defines two successively stronger partial orders on $M_n$, which we denote as $\M_n^C:=(M_n,\leq_C)$ and $\M_n^T:=(M_n,\leq_T)$ and will define in Section \ref{section:Motzkin defs}. The partial order $\M_n^T$ was introduced by Fang \cite{Fang} as an analogue to classical Tamari lattices, which have garnered much interest from researchers in combinatorics, group theory, theoretical computer science, algebraic geometry, and algebraic topology \cite{Combe, Csar, Early, Fang, Geyer, Huang, Knuth2, Loday, Pallo, Tamari}. Defant referred to $\M_n^T$ as the $n^{\text{th}}$ \emph{Motzkin-Tamari poset}, and we will call $\M_n^C$ the $n^\text{th}$ \emph{Motzkin-Defant poset} analogously.

\subsection{Main Results}

In \cite{Defant}, Defant provides generating functions and asymptotics for the number of valid hook configurations of length-$n$ permutations avoiding most specified length-3 patterns. However, he was unable to compute either in the 312-avoiding case. Defant constructed a bijection $\widehat\LL_n$ (pronounced ``double lambda'') from valid hook configurations of 312-avoiding permutations to intervals in the Motzkin-Defant poset $\M_{n-1}^C$ and conjectured that the numbers of these intervals could be enumerated by the binomial transform of certain lattice paths in $\mathbb N^2$ studied in \cite{brs}. We prove Defant's conjecture in Section \ref{section:312} by defining a related class of lattice paths with the desired cardinality and providing a bijection from intervals in the Motzkin-Defant poset to this class of paths. Moreover, we use the results from \cite{brs} to compute asymptotics in the 312-case and to show that Defant's goal of producing an explicit generating function is untenable in this case, as the generating function is not algebraic or even $D$-finite. Even if we ignore valid hook configurations, this bijection connects intervals in posets of Motzkin paths and lattice paths in $\mathbb N^2$, which are more classical objects.

Defant additionally shows in \cite{Defant} that the number of valid hook configurations of 132-avoiding length-$n$ permutations and the number of intervals in the Motzkin-Tamari poset $\M_{n-1}^T$ are the same, by comparing their generating functions. He asks for an explicit bijection between the two sets. In Sections \ref{section:swl} through 7 we construct such a bijection by providing an injection from valid hook configurations of 132-avoiding permutations to valid hook configurations of 312-avoiding permutations and composing the injection with Defant's bijection $\widehat\LL_n$. In addition to solving Defant's problem, this injection suggests a more general result about inequalities among valid hook configurations avoiding different inequalities, which we posit in Section \ref{section:future work}. We construct the injection by generalizing a map $\swl$ introduced by Defant \cite{hannaspaper}. In order to study our map, we build up a number of generally applicable results about $\swl$ in Section \ref{section:swl}. In addition, we present an equivalent but more intuitive and versatile definition of a valid hook configuration in Section \ref{section:defs}. This new theory of valid hook configurations is also generally applicable: in addition to simplifying the analysis of our injection, it would have, for example, allowed Defant to define $\widehat\LL_n$ more easily.

\section{Permutations and Valid Hook Configurations}\label{section:defs}

Let $S_n$ be the set of permutations $\pi:[n]\to[n]$. We will treat a permutation $\pi\in S_n$ as a word $\pi_1\cdots\pi_n$ in which each term $\pi_i$ is a distinct element of $[n]$. For $\sigma\in S_k$ and $\pi\in S_n$, we say that a length-$k$ subsequence $\pi_{a_1}\cdots\pi_{a_k}$ \emph{matches} the pattern $\sigma$ if the indices $a_i$ are increasing and $\pi_{a_i}>\pi_{a_j}$ if and only if $\sigma_i>\sigma_j$. We say $\pi$ \emph{avoids} $\sigma$ if no such subsequence of $\pi$ matches $\sigma$ and we denote the set of $\sigma$-avoiding permutations $\pi\in S_n$ by $\Av_n(\sigma)$. The \emph{plot} of $\pi$ is defined to be the set of points $\{(i,\pi_i):i\in[n]\}$. A \emph{descent} occurs when $\pi_i>\pi_{i+1}$; we will call $\pi_i$ the \emph{descent top} and $\pi_{i+1}$ the \emph{descent bottom}.

A \emph{hook} on $\pi$ is constructed by drawing a vertical line up from a point $(i,\pi_i)$ and then drawing a horizontal line right to a point $(j,\pi_j)$. This requires both $i<j$ and $\pi_i<\pi_j$. We call $(i,\pi_i)$ and $(j,\pi_j)$ the \emph{southwest}  and \emph{northeast} endpoints of the hook, respectively.

\begin{figure}[h]
\begin{center}
\begin{tikzpicture}[scale=.8]
\fill (1, 3) circle (.9mm) node[above left] {$(1,3)$} (2, 4) circle (.9mm) node[below right] {$(2,4)$} (3, 1) circle (.9mm) node[above left] {$(3,1)$} (4, 5) circle (.9mm) node[above left] {$(4,5)$} (5,2) circle (.9mm) node[above left] {$(5,2)$};
\draw (1,3) --(1,5) -- (4,5);
\end{tikzpicture}
\caption{The plot of $\pi=34152$ with a hook from $(1,3)$ to $(4,5)$.}
\end{center}  
\end{figure}
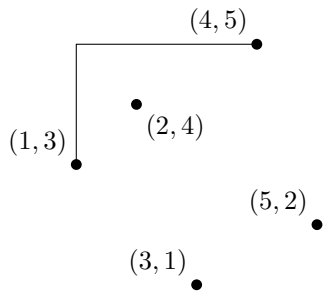

\begin{defn}\label{def:VHC}
A \emph{valid hook configuration} on a permutation $\pi$ is a set of hooks such that
\begin{enumerate}[label=(\roman*)]
\item The set of southwest hook endpoints is exactly the set of descent tops.
\item No point in the plot of $\pi$ may lie above a hook.
\item Hooks cannot intersect each other except at their endpoints.
\end{enumerate}
Examples of valid and invalid hook configurations are given in Figures \ref{fig:valid HC} and \ref{fig:invalid HC}.
\end{defn}

\begin{figure}[h]\begin{center}
\begin{tikzpicture}[scale=.8]
	\coordinate (P1) at (1,3);
	\coordinate (P2) at (2,2);
	\coordinate (P3) at (3,1);
	\coordinate (P4) at (4,5);
	\coordinate (P5) at (5,6);
	\coordinate (P6) at (6,4);
	\coordinate (P7) at (7,7);
	\foreach \pt in {P1, P2, P3, P4, P5, P6, P7}
		\fill (\pt) circle (.9mm);
	\foreach \a/\b in {P1/P5, P2/P4, P5/P7}
		\draw (\a) -- (\a)|-(\b) -- (\b);
\end{tikzpicture}
\caption{A valid hook configuration on the permutation 3215647.}
\label{fig:valid HC}
\end{center}\end{figure}
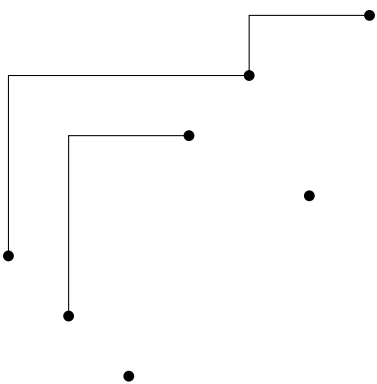
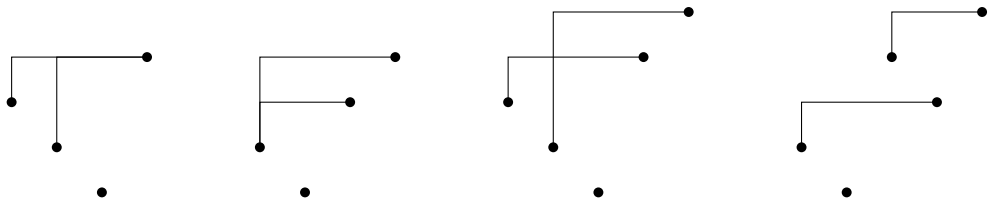
\begin{figure}[h]\begin{center}
\begin{tikzpicture}[scale=.6]
	\coordinate (S2) at (0,0);
	\coordinate (S3) at (5.5,0);
	\coordinate (S1) at ($(S3) + (5.5,0)$);
	\coordinate (S4) at ($(S1) + (6.5,0)$);
	\coordinate (P1) at (1,3);
	\coordinate (P2) at (2,2);
	\coordinate (P3) at (3,1);
	\coordinate (P4) at (4,4);
	\coordinate (P5) at (5,5);
	\coordinate (Q1) at (1,2);
	\coordinate (Q2) at (2,1);
	\coordinate (Q3) at (3,3);
	\coordinate (Q4) at (4,4);
	\coordinate (R3) at (3,4);
	\coordinate (R4) at (4,3);
	\coordinate (R5) at (5,5);
	\foreach \pt in {P1, P2, P3, P4, P5}
		\fill ($ (\pt) + (S1) $) circle (1.1mm);
	\foreach \pt in {P1, P2, P3, P4}
		\fill ($ (\pt) + (S2) $) circle (1.1mm);
	\foreach \pt in {Q1, Q2, Q3, Q4}
		\fill ($ (\pt) + (S3) $) circle (1.1mm);
	\foreach \pt in {Q1, Q2, R3, R4, R5}
		\fill ($ (\pt) + (S4) $) circle (1.1mm);
	\foreach \a/\b/\s in {P1/P4/S1, P2/P5/S1, P1/P4/S2, P2/P4/S2, Q1/Q3/S3, Q1/Q4/S3, Q1/R4/S4, R3/R5/S4}
		\draw ($(\a)+(\s)$) -- ($(\a)+(\s)$)|-($(\b)+(\s)$) -- ($(\b)+(\s)$);
\end{tikzpicture}
\caption{Hook configurations on permutations failing conditions (ii) or (iii) of Definition \ref{def:VHC}.}
\label{fig:invalid HC}
\end{center}\end{figure}

For $S\subset S_n$, we write $\VHC(S)$ for the set of valid hook configurations on permutations in $S$. The following two propositions describe our new formulation of valid hook configurations and how to apply it.

\begin{prop}
Let $\pi$ be a permutation and $V$ a set of points in the plot of $\pi$. Then there is at most one valid hook configuration whose set of northeast endpoints is exactly $V$.
\end{prop}

\begin{proof}
Let $U$ be the set of descent tops of $\pi$. List the points of $U\sqcup V:=U\times\{0\}\cup V\times\{1\}$ from left to right to get a sequence $(P_1,\delta_1),\ldots,(P_{|U|+|V|},\delta_{|U|+|V|})$ with $(P_i, 1)$ listed before $(P_i, 0)$ when $P_i\in U\cap V$. We can turn this sequence into a sequence of open and close parentheses by $(P_i,\delta_i)\mapsto\left\{\begin{array}{cl}\code{`('}&\text{if }\delta_i=0\\\code{`)'}&\text{if }\delta_i=1\end{array}\right.$. Every valid hook configuration on $\pi$ with $V$ the set of northeast endpoints gives a different matching of left and right parentheses. Furthermore, we show this parenthesis matching must be \emph{balanced}, meaning that a matched pair \texttt{()} encloses either both or neither parentheses from any other matched pair. 

If the matching were not balanced, we would have two improperly matched pairs of parentheses, as shown in Figure \ref{fig:unbalanced parentheses}.
\begin{figure}[h]\begin{center}
\begin{tikzpicture}
\coordinate (A) at (1,1);
\coordinate (B) at (3, 2.3);
\coordinate (C) at (2, 2.9);
\coordinate (D) at  (4, 3.5);
\coordinate(Ca) at ($ (C|-A)!.5!(C|-B) $);
\draw[dashed] (Ca) -- (C);
\draw (A) -- (A|-B) node[left] {$H_1$} -- (B);
\draw (C) -- (C|-D) node[left] {$H_2$}-- (D);
\fill (A) circle (.7mm) (B) circle (.7mm) (D) circle (.7mm);
\fill[black!40] (C) circle (.7mm) (Ca) circle(.7mm);
\draw(A) node[below] {\hphantom{$_1$}\code{(}$_1$};
\draw(A-|B) node[below] {\hphantom{$_1$}\code{)}$_1$};
\draw(A-|C) node[below] {\hphantom{$_2$}\code{(}$_2$};
\draw(A-|D) node[below] {\hphantom{$_2$}\code{)}$_2$};
\end{tikzpicture}
\caption{Hooks corresponding to two improperly matched pairs of parentheses.}
\label{fig:unbalanced parentheses}
\end{center}\end{figure}
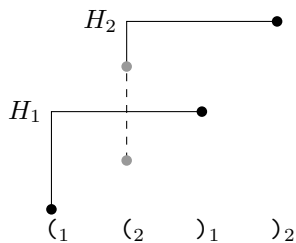
If the southwest endpoint of hook $H_2$ lies above hook $H_1$ then condition (ii) of Definition \ref{def:VHC} is violated, and if it lies below $H_1$ then condition (iii) is violated. Hence, this matching is balanced.

However, it is well-known that any sequence of parentheses has at most one balanced matching, so it follows that there can be at most one valid hook configuration on $\pi$ with set of northeast endpoints $V$.
\end{proof}

We will typically refer to a valid hook configuration by the set of its northeast endpoints, and often write $(\pi, V)$ for a valid hook configuration on $\pi$ with $V$ the set of northeast endpoints.

\begin{prop}\label{prop:new VHC}
Assume $V$ is a set of points in the plot of a permutation $\pi$ and there is a bijection $\phi$ from descent tops in $\pi$ to $V$ such that if $\phi(i,\pi_i)=(j,\pi_j)$, then
\begin{enumerate}[label=(\roman*)]
\item We have $j>i$ and $\pi_j>\pi_i$.
\item There is no $k$ with $i<k<j$ and $\pi_k>\pi_j$.
\end{enumerate}
Then $V$ is a valid hook configuration on $\pi$.
\end{prop}

\begin{proof}
We can think of $\phi$ as a set of hooks on $\pi$ because of condition (i). Condition (ii) of the proposition implies that there are no points above any of these hooks, so everything but condition (ii) of Definition \ref{def:VHC} is satisfied. We induct on the number of hook intersections in $\phi$. When there are none, this set of hooks is a valid hook configuration.

Now assume two hooks $H_1$ and $H_2$ of $\phi$ intersect. Let $H_1=ABC$ and $H_2=DEF$ as shown at left in Figure \ref{fig:hooks to more hooks}.
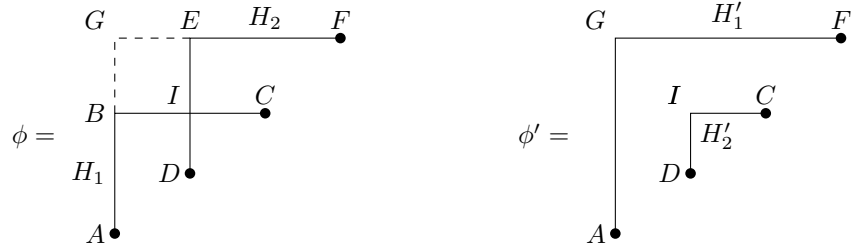
\begin{figure}[h]
\begin{center}
$\phi=\begin{tikzpicture}[baseline=(O)]
\coordinate (A) at (1,1.4);
\coordinate (D) at (2,2.2);
\coordinate (C) at (3, 3);
\coordinate (B) at (A|-C);
\coordinate (F) at (4, 4);
\coordinate (E) at (D|-F);
\coordinate (G) at (A|-F);
\coordinate (I) at (D|-C);
\coordinate (O) at ($ (C)!.5!(D) $);
\fill (A) circle (.7mm) (C) circle (.7mm) (D) circle (.7mm) (F) circle (.7mm);
\draw (A) node[left] {$A$}-- (B) node[pos=.5, left] {$H_1$}  -- (C) node[above] {$C$};
\draw (F) node[above] {$F$}-- (E) node[pos=.5, above] {$H_2$} -- (D) node[left] {$D$};
\draw [dashed] (B) node[left] {$B$} -- (G) node[above left] {$G$} -- (E) node[above] {$E$};
\draw (I) node[above left] {$I$};
\end{tikzpicture}
\qquad\qquad\qquad
\phi'=\begin{tikzpicture}[baseline=(O)]
\coordinate (A) at (1,1.4);
\coordinate (D) at (2,2.2);
\coordinate (C) at (3, 3);
\coordinate (B) at (A|-C);
\coordinate (F) at (4, 4);
\coordinate (E) at (D|-F);
\coordinate (G) at (A|-F);
\coordinate (I) at (D|-C);
\coordinate (O) at ($ (C)!.5!(D) $);
\fill (A) circle (.7mm) (C) circle (.7mm) (D) circle (.7mm) (F) circle (.7mm);
\draw (D) node[left] {$D$} -- (I) node[above left]{$I$} node[below right]{$H'_2$} -- (C) node[above]{$C$};
\draw (A) node[left] {$A$} -- (G) node[above left] {$G$} -- (F) node[above] {$F$} node[pos=.5, above] {$H'_1$};
\draw (I) node[above left] {$I$};
\end{tikzpicture}$
\end{center}
\caption{Original hooks $H_1$ and $H_2$ at left and replacement hooks $H'_1$ and $H'_2$ at right.}
\label{fig:hooks to more hooks}
\end{figure}
Let $\phi'$ be the set of hooks achieved from $\phi$ by replacing $H_1$ and $H_2$ with $H'_1=AGF$ and $H'_2=DIC$, as shown at right in the same figure. Any point above $H'_1$ or $H'_2$ would also be a point above $H_1$ or $H_2$, so there are no points above hooks in $\phi'$. To apply the inductive hypothesis we must show that the number of hook intersections in $\phi'$ is strictly less than the number of hook intersections in $\phi$. Consider a hook $H_3\in\phi$ distinct from $H_1$ and $H_2$. If $H_3$ intersects $GE$ or $GB$ then it must intersect $BI$ or $EI$ respectively, or else there would be a hook endpoint in square $GEIB$ above $H_1$. It follows that $H_3$ intersects $H_1$ and $H_2$ at least as many times as it intersects $H'_1$ and $H'_2$. Furthermore, the hook intersection between $H_1$ and $H_2$ has been removed in $\phi'$. It follows that $\phi'$ has strictly fewer hook intersections than $\phi$, and we can apply the inductive hypothesis to show that $V$ is a valid hook configuration on $\pi$.
\end{proof}

\section{Motzkin Intervals}\label{section:Motzkin defs}

We define the two posets $\M_n^C$ and $\M_n^T$ on $M_n$. Recall that the Motzkin-Stanley lattice has the relation $P\leq_SQ$ if $P$ is below or equal to $Q$.

\begin{defn}\label{def:y-displacement}
If $s$ is a step in a Motzkin path, let $\delta(s)$ be the $y$-displacement of that step. That is, $\delta(U)=1$, $\delta(E)=0$, and $\delta(D)=-1$.
\end{defn}

\begin{defn}\label{def:class}
The \emph{class} of a path $P\in M_n$, denoted $\cl(P)$, is the subsequence of $P$ consisting of $U$'s and $E$'s.
\end{defn}

\begin{defn}\label{def:long}
Write $P\in M_n$ as $P=X_1D^{\gamma_1}\cdots X_\ell D^{\gamma_\ell}$ with $X_i\in\{U,E\}$. For $i\in[\ell]$, we define $\lng_i(P)$ to be the length of the shortest consecutive substring of $P$ starting at $X_i$ that forms a Motzkin sub-path.
\end{defn}

As an example, if we take $P=UDEUEUDD$ pictured in Figure \ref{fig:class and longs} then $\cl(P)=UEUEU$ and we have $\lng_1(P)=2$, $\lng_2(P)=1$, $\lng_3(P)=5$, $\lng_4(P)=1$, and $\lng_5(P)=2$.
\begin{figure}[h]
\begin{center}
\begin{tikzpicture}[scale=.5]
\draw[step=1, gray, thin] (0,0) grid (8,2);
\draw[thick] (0,0) -- ++ (1,1) -- ++ (1, -1)--++(1,0)--++(1,1)--++(1,0)--++(1,1)--++(1,-1)--++(1,-1);
\end{tikzpicture}
\end{center}  
\caption{The Motzkin path $P=UDEUEUDD$.}
\label{fig:class and longs}
\end{figure}
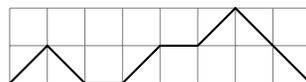
It is also worth noting that if $X_i=E$ then $\lng_i(P)=1$ and if $X_i=U$ then $\lng_i(P)\geq 2$. This definition is different from the one provided in \cite{Defant}, but it is easier to work with and does not affect the definition of the Motzkin-Tamari poset below.

\begin{defn}\label{def:colin poset}
We say $P\leq_CQ$ if $\cl(P)=\cl(Q)$ and $P\leq_SQ$. We write $\M_n^C:=(M_n,\leq_C)$ and call this the $n^\text{th}$ \emph{Motzkin-Defant poset}.
\end{defn}

\begin{defn}\label{def:tamari poset}
If $\cl(P)=\cl(Q)$ then we can write $P=X_1D^{\gamma_1}\cdots X_{\ell}D^{\gamma_\ell}$ and $Q=X_1D^{\gamma'_1}\cdots X_\ell D^{\gamma'_\ell}$ as in Definition \ref{def:long}. We say $P\leq_TQ$ if $\cl(P)=\cl(Q)$ and $\lng_i(P)\leq\lng_i(Q)$ for all $i\in[\ell]$. We write $\M_n^T:=(M_n,\leq_T)$ and call this the $n^\text{th}$ \emph{Motzkin-Tamari poset}.
\end{defn}

\begin{defn}\label{def:interval}
An \emph{interval} in a poset $\mathcal P$ is an ordered pair of elements $(x_1,x_2)\in\mathcal P^2$ for which $x_1\leq_{\mathcal P}x_2$. We write $\Int(\mathcal P)$ for the set of intervals in $\mathcal P$.
\end{defn}

We can now notate the domain and range of Defant's bijection $\widehat\LL_n$, which we will describe in Section \ref{section:define LL}. It is a map $\VHC(\Av_n(312))\to\Int(\M_{n-1}^C)$.

\section{312-Avoiding Valid Hook Configurations}\label{section:312}

This section enumerates $|\VHC(\Av_n(312))|$ using the following lattice paths.

\begin{defn}\label{def:mathfrak w}
Let $\w(k)$ be the number of lattice paths of length $k$ in the first quadrant with steps in $\{(-1, 0), (-1, 1), (0, -1), (0, 1), (1, -1)\}$ starting and ending at the origin, as in the OEIS sequence A151347.
\end{defn}

Bostan, Raschel, and Salvy studied $\w(k)$ in \cite{brs} and computed its asymptotic growth to be $\w(k)=\Theta\left(\frac{4.729032^k}{k^{4.514931}}\right)$, where both decimals are truncated approximations. They additionally showed that the generating function $F(x)=\sum_{k=0}^\infty\w(k)x^k$ is not \emph{$D$-finite}, meaning that there is no linear dependence among the generating functions $F(x),F'(x),F''(x),\ldots$ with coefficients in $\mathbb R[x]$. All algebraic functions are $D$-finite, as well as some ``nice'' transcendental functions like $\sin$ and $\cos$. We will prove a conjecture of Defant's that $|\VHC(\Av_n(312))|=\sum_{k=0}^{n-1}\binom{n-1}k\w(k)$. We will further use the results of \cite{brs} to analyze the asymptotics of $|\VHC(\Av_n(312))|$ and show that the generating function $G(x)=\sum_{n\geq 1}|\VHC(\Av_n(312))|x^n$ is not $D$-finite, explaining why Defant was unable to find an explicit algebraic form for it.

\begin{defn}
Let $\N_n$ be the set of pairs $(X,Y)$ of Motzkin paths of length $n$ whose $i^\text{th}$ coordinates $(X_i,Y_i)$ are forbidden from being $(D,D)$, $(U,U)$, or $(U,E)$. That is,
\[\N_n=\{(X,Y)\in M_n^2:(X_i,Y_i)\in \{(D,E), (D,U), (E,D), (E,E), (E,U), (U,D)\}\text{ for each }i\in[n]\}.\]
\end{defn}

\begin{prop}\label{prop:size of nn}
We have $|\N_n|=\sum_{k=0}^n\binom nk\mathfrak{w}(k)$.
\end{prop}

\begin{proof}
Convert $(X,Y)\in\N_n$ into a lattice path with $i^\text{th}$ step $(\delta(X_i),\delta(Y_i))$. This is the same as a lattice path of length $n-k$ as in Definition \ref{def:mathfrak w}, where $k$ is the number of $(0,0)$ steps, together with a set of indices in $\binom{[n]}k$ identifying which steps are $(0,0)$.
\end{proof}

We will show that $|\VHC(\Av_{n+1}(312))|=|\N_n|$. To construct a bijection between the two sets we begin with the bijection $\widehat\LL_{n+1}:\VHC(\Av_{n+1}(312))\to\Int(\M_{n}^C)$ constructed by Defant in \cite{Defant}. To complete the bijection into $\N_n$, we construct a bijection $\varphi_n:\Int(\M_n^C)\to\N_n$.

\begin{definition}\label{def:Dyck prefix}
A \emph{Dyck prefix} of length $n$ is a sequence $A=(A_1,\ldots,A_n)$ of $u$'s and $d$'s such that for each $i\leq n$, there are at least as many $u$'s as $d$'s in $A_1,\ldots,A_i$. If $A$ and $A'$ are two Dyck prefixes of the same length, we say $A\geq A'$ if $A_1,\ldots,A_i$ has at least as many $u$'s as $A'_1,\ldots,A'_i$ for each $i\leq n$.
\end{definition}

We can also interpret Definition \ref{def:Dyck prefix} visually. Letting $u$'s and $d$'s correspond to lattice steps of $(1,1)$ and $(1,-1)$ respectively, a Dyck prefix is a lattice path with these steps that starts on the $x$-axis and never goes below it; or a sequence of $u$'s and $d$'s that is the prefix of some Dyck path. In this interpretation, $A\geq A'$ if the lattice path corresponding to $A$ lies above the lattice path corresponding to $A'$.

\begin{definition}
The \emph{support} of a length-$n$ Motzkin path $P$, written $\overline{P}$, is a Dyck prefix of length $n$ with $\overline P_i=u$ if $P_i$ is $U$ or  $E$ and $\overline P_i=d$ if $P_i$ is $D$.
\end{definition}

It is clear that the support $\overline{P}$ of a Motzkin path is in fact a Dyck prefix because, as a lattice path it must lie above $P$, which in turn lies above the $x$-axis. We also note that every Motzkin path can be recovered from its class and support, although not every (class, support) pair represents a valid Motzkin path.

\begin{lemma}\label{lemma:describe larger paths}
For a fixed Motzkin path $P$ of length $n$, the Motzkin paths $Q\geq_CP$ are in bijection with Dyck prefixes $A\geq\overline P$ with the same numbers of $u$'s and $d$'s as $\overline P$. The bijection takes $Q$ to its support, $\overline Q$.
\end{lemma}

\begin{proof}
We begin by showing that if $Q\geq_CP$ then $\overline Q\geq\overline P$. If $Q\geq_CP$ then each prefix $P_1\cdots P_k$ must have at least as many down steps as the corresponding prefix $Q_1\cdots Q_k$, because $P$ and $Q$ have the same class. It follows that $\overline{P_1\cdots P_k}$ also has at least as many $d$ steps as $\overline{Q_1\cdots Q_k}$ for each $k$, which is exactly the condition for $\overline P\leq \overline Q$.

The map $Q\mapsto\overline Q$ is injective because all paths $Q\geq_CP$ share the same class, and a Motzkin path is uniquely determined by its class and support.
To see surjectivity, note that we can get every such Dyck prefix $A\geq\overline P$ from $\overline P$ by an appropriate sequence of flips that change the substring $du$ to $ud$. Applying the corresponding flips to $P$ gives a Motzkin path $Q\geq_CP$ because both the flips $DU\to UD$ and $DE\to ED$ preserve the class of the Motzkin path and make it higher, and it is clear that this path $Q$ has $\overline Q=A$.
\end{proof}

We now have enough information to construct our bijection from $\Int(\M_n^C)$ to $\N_n$.

\begin{theorem}
There is a bijection $\varphi_n:\Int(\M_n^C)\to\N_n$ given by $\varphi(P, Q)=(X,Y)$ where
\[
X_i=\left\{\begin{array}{cl}
U&\text{if $\overline{P_i}=d$ and $\overline{Q_i}=u$}
\\D&\text{if $\overline{P_i}=u$ and $\overline{Q_i}=d$}
\\E&\text{if $\overline{P_i}=\overline{Q_i}$}
\end{array}\right.
\]
and $Y=P$. 
\end{theorem}

\begin{proof}
We begin by showing that this map is well-defined; that is, that its image is in fact in $\N_n$. It is impossible for $(X_i,Y_i)$ to be $(U,U)$, $(U,E)$, or $(D,D)$ by the definition of $\varphi_n$. Moreover we note that the height of $X$ after $i$ steps is the difference between the heights of $\overline Q$ and $\overline P$ after $i$ steps, treating both as lattice paths starting at the same height. Because $\overline P\leq\overline Q$, it follows that $X$ never goes below the horizontal; similarly, $X$ ends on the horizontal because $\overline P$ and $\overline Q$ must end at the same height, having the same numbers of $u$'s and $d$'s. Hence $X$ is a Motzkin path.

To show that $\varphi_n$ is a bijection, we construct an inverse $\phi_n:\mathcal N_n\to\Int(\M_n^C)$ and prove that $\varphi_n\circ\phi_n$ and $\phi_n\circ\varphi_n$ are both identities. Let $\phi_n(X,Y)=(P,Q)$ where $P=Y$ and $Q$ is the Motzkin path with $Q\geq_C P$ and support 
\[
\overline{Q_i}=\left\{\begin{array}{cl}
u&\text{if }X_i=U
\\d&\text{if }X_i=D
\\\overline{Y_i}&\text{if }X_i=E
\end{array}\right.
.\]
To show that $\phi_n$ is well-defined, we must show that $\overline{Q}\geq\overline P$ and that $\overline Q$ and $\overline P$ have the same numbers of $u$'s and $d$'s, as required by Lemma \ref{lemma:describe larger paths}. We can pair up the $U$'s and $D$'s in $X$ so that the $D$ in each pair occurs after the $U$ --- one way to do this is to convert $U$'s and $D$'s into \texttt{(}'s and \texttt{)}'s and then take a balanced parenthesis matching. We start with the Dyck prefix $\overline P=\overline Y$ and for each of these pairs $(X_i,X_j)=(U,D)$, we change the $i^\text{th}$ term of the Dyck prefix to $u$ and the $j^\text{th}$ term to $d$ (note that these terms were originally $d$ and $u$, respectively, because $(X,Y)\in\mathcal N_n$). Every such operation gives a Dyck prefix with the same numbers of $u$'s and $d$'s that is greater than the previous Dyck prefix, and after all these operations are performed, we get the desired support $\overline{Q}$. Hence the condition of Lemma \ref{lemma:describe larger paths} is satisfied, and $\phi_n$ is well-defined.

We now show that $\varphi_n\circ\phi_n$ and $\phi_n\circ\varphi_n$ are both identities, starting with $\varphi_n\circ\phi_n$. Fix $(X,Y)\in\N_n$ and let $(Y, Q)=\phi_n(X,Y)$ and $(X',Y)=\varphi_n\circ\phi_n(X,Y)$. We show that $X'_i=X_i$ by casework. If $X_i=U$ then $\overline{Q_i}=u$ and $(X,Y)\in\N_n$ implies that $(X_i,Y_i)$ is not $(U,U)$ or $(U,E)$, so $\overline{Y_i}=d$. Hence $X'_i=U$ if $X_i=U$. Similarly, if $X_i=D$ then $\overline{Q_i}=d$ and because $(X,Y)\in\N_n$, we have $(X_i,Y_i)\neq(D,D)$, so $\overline{Y_i}=u$. It follows that $X'_i=D$ if $X_i=D$. In the case that $X_i=E$, we have $\overline{Q_i}=\overline{Y_i}$, so $X'_i=E$. Therefore, $X'=X$, and it follows that $\varphi\circ\varphi'$ is the identity.

Lastly, we consider $\phi_n\circ\varphi_n$. Fix $(P,Q)\in\M_n^C$ and let $(X,P)=\varphi_n(P,Q)$ and $(P,Q'')=\phi_n\circ\varphi_n(P,Q)$. We claim that $\overline{Q'}=\overline{Q}$, by casework on each step $X_i$. If $X_i=U$ then $\overline{Q_i}=u$ and $\overline{Q'_i}=u$; similarly if $X_i=D$ then $\overline{Q_i}=d$ and $\overline{Q'_i}=d$. If $X_i=E$ then $\overline{Q_i}=\overline{P_i}$ and $\overline{Q'_i}=\overline{P_i}$. Hence, $\overline Q=\overline{Q'}$. Because $Q$ and $Q'$ have the same support and both have the same class as $P$, it follows that $Q=Q'$. Hence $\phi_n\circ\varphi_n$ is the identity.
\end{proof}

Combining $\varphi_{n-1}$ with $\widehat\LL_n$ completes the proof of Defant's conjecture.

\begin{cor}\label{cor:count with w}
We have $|\VHC(\Av_n(312))|=\sum_{k=0}^{n-1}\binom{n-1}k\mathfrak w(k)$.
\end{cor}

\begin{proof}
Composing the bijections $\widehat\LL_{n}:\VHC(\Av_n(312))\to\Int(\M_{n-1}^C)$ and $\varphi_{n-1}:\Int(\M_{n-1}^C)\to\N_{n-1}$ shows that $|\VHC(\Av_n(312))|=|\N_{n-1}|$. Proposition \ref{prop:size of nn} completes the proof.
\end{proof}

We can now use the results of Bostan, Raschel, and Salvy in \cite{brs} to analyze the generating function and asymptotics of the sequence $|\VHC(\Av_n(312))|$.

\begin{prop}\label{prop:not D finite}
The generating function $G(x)=\sum_{n=1}^\infty|\VHC(\Av_n(312))|x^n$ is not $D$-finite.
\end{prop}

\begin{proof}
Let $F(x)$ be the generating function $\sum_{k=0}^\infty\w(x)x^k$ which Bostan, Rachel, and Salvy showed to be non-$D$-finite. Because $\frac{G(x)}x$ is the binomial transform of $F(x)$, we can write $\frac{G(x)}x=\frac{1}{1-x}F\left(\frac{x}{1-x}\right)$. Setting $y=\frac{x}{1-x}$ yields $F(y)=\frac 1yG\left(\frac y{1+y}\right)$. Theorem 6.4.10 of \cite{Enumerative Combinatorics} states that the composition $G\circ h$ of a $D$-finite function $G$ and algebraic function $h$ is also $D$-finite, and the product of $D$-finite functions is also $D$-finite. Hence, if $G$ were $D$-finite then $F(y)=\frac 1yG\left(\frac y{1+y}\right)$ would also be $D$-finite, which is a contradiction.
\end{proof}

\begin{theorem}\label{thm:312 asymtotics}
Let $\alpha\approx 4.515$ and $\beta\approx 4.729$ be the values for which $\w(k)=\Theta\left(\frac{\beta^k}{k^\alpha}\right)$, as determined in \cite{brs}. Then $|\VHC(\Av_n(312))|=\Theta\left(\frac{(\beta+1)^n}{n^\alpha}\right)\approx\Theta\left(\frac{5.729^n}{n^{4.515}}\right)$.
\end{theorem}

The proof of this theorem relies on the following lemma.

\begin{lemma}\label{lemma:alphas and betas}
Fix $\alpha\geq 0$ and $\beta\geq 1$. As $n\to\infty$, we have
\[\sum_{k=0}^n\binom nk\frac{\beta^k}{k^\alpha}\sim(\beta+1)^n\left(\frac{\beta n}{\beta+1}\right)^{-\alpha}.\qedhere\]
\end{lemma}

\begin{proof}
We show that almost all of the sum comes from values of $k$ close to $\frac{\beta n}{\beta+1}$. Let $X=\sum_{i=1}^nX_i$ be the sum of $n$ independent Bernoulli variables, with $X_i=1$ with probability $\frac{\beta}{\beta+1}$ for each $i$. Let $\mu=\E[X]=\frac{\beta n}{\beta+1}$ and let $\delta=n^{-1/3}$. A Chernoff bound says that
\[
\Pr[|X-\mu|\geq\delta\mu]\leq 2e^{-\mu\delta^2/3}.
\]
This expands to
\[
\sum_{\substack{k<(1-\delta)\mu\\\text{or }k>(1+\delta)\mu}}\binom nk\frac{\beta^k}{(\beta+1)^n}\leq 2\exp\left(-\frac{\beta n^{1/3}}{3(\beta+1)}\right).
\]

We can use this Chernoff bound to bound the original sum. The original sum splits as
\[
\sum_{k=0}^n\binom nk\frac{\beta^k}{k^\alpha}=\sum_{|k-\mu|\leq\delta\mu}\binom nk\frac{\beta^k}{k^\alpha}+\sum_{\substack{k<(1-\delta)\mu\\\text{or }k>(1+\delta)\mu}}\binom nk\frac{\beta^k}{k^\alpha}
\]
and we upper-bound the two terms separately. We have
\begin{align*}
\sum_{|k-\mu|\leq\delta\mu}\binom nk\frac{\beta^k}{k^\alpha}&\leq((1-\delta)\mu)^{-\alpha}\sum_{|k-\mu|\leq\delta\mu}\binom nk\beta^k\leq((1-\delta)\mu)^{-\alpha}\sum_{k=0}^n\binom nk\beta^k
\\&=((1-\delta)\mu)^{-\alpha}(\beta+1)^n\sim\mu^{-\alpha}(\beta+1)^n,\text{ \ and}
\\
\sum_{\substack{k<(1-\delta)\mu\\\text{or }k>(1+\delta)\mu}}\binom nk\frac{\beta^k}{k^\alpha}&\leq\sum_{\substack{k<(1-\delta)\mu\\\text{or }k>(1+\delta)\mu}}\binom nk\beta^k\leq 2(\beta+1)^n\exp\left(-\frac{\beta n^{1/3}}{3(\beta+1)}\right).
\end{align*}
However, $\mu^{-\alpha}=e^{-\Theta(\ln n)}\gg e^{-\Theta(n^{1/3})}$, so the upper bound is dominated by the first term. Hence,
\[
\sum_{k=0}^n\binom nk\frac{\beta^k}{k^\alpha}\lesssim(\beta+1)^n\mu^{-\alpha}.
\]

For the lower bound, we consider only the first term. Using the Chernoff bound again, we have
\begin{align*}
\sum_{k=0}^n\binom nk\frac{\beta^k}{k^\alpha}
&\geq\sum_{|k-\mu|\leq\delta\mu}\binom nk\frac{\beta^k}{k^\alpha}\geq((1+\delta)\mu)^{-\alpha}\sum_{|k-\mu|\leq\delta\mu}\binom nk\beta^k
\\&\geq((1+\delta)\mu)^{-\alpha}\left(1-2\exp\left(-\frac{\beta n^{1/3}}{3(\beta+1)}\right)\right)(\beta+1)^n
\sim(\beta+1)^n\mu^{-\alpha}.
\end{align*}
Therefore,
\[\sum_{k=0}^n\binom nk\frac{\beta^k}{k^\alpha}\sim(\beta+1)^n\mu^{-\alpha}=(\beta+1)^n\left(\frac{\beta n}{\beta+1}\right)^{-\alpha},\]
completing the proof.
\end{proof}

We can now compute the asymtotics of $\VHC(\Av_n(312))$. 

\begin{proof}[Proof of Theorem \ref{thm:312 asymtotics}]
Let $c$ be the constant such that $\w(k)\sim c\frac{\beta^k}{k^\alpha}$. Because $\w(k)$ goes to $\infty$ as $k$ goes to $\infty$, we have that $\sum_{k=0}^{n-1}\binom {n-1}k\w(k)\sim c\sum_{k=0}^{n-1}\binom {n-1}k\frac{\beta^k}{k^\alpha}$ as $n$ goes to $\infty$. Combining this with Corollary \ref{cor:count with w} and Lemma \ref{lemma:alphas and betas} yields that $|\VHC(\Av_n(312))|\sim c(\beta+1)^{n-1}\left(\frac{\beta (n-1)}{\beta+1}\right)^{-\alpha}=\Theta\left(\frac{(\beta+1)^n}{n^\alpha}\right).$
\end{proof}

\section{Sliding operators}\label{section:swl}

The focus of the next two sections will be constructing an injection \[\W_n:\VHC(\Av_n(132))\to \VHC(\Av_n(312))\] such that its composition with Defant's bijection $\widehat{\LL}_n:\VHC(\Av_n(312))\to\Int(\M_{n-1}^C)$ is surjective onto intervals of the subposet $\M_{n-1}^T$. To compute $\varphi(\pi, V)$, we will first define the image permutation as $\swl(\pi)$ (defined below) and then use $V$ to find an image VHC on $\swl(\pi)$.

We start by defining the sliding operators $\swl$ and $\swr$. We define $\swl_i$ to take all points southwest of the point with height $i$ and slide them left of all points northwest of the point with height $i$ --- hence the name ``southwest left.''

\begin{figure}[h]\begin{center}
\begin{tikzpicture}[scale=.5]
\coordinate (s) at (17,0);
\coordinate (i1) at (7,5);
\coordinate (i2) at ($ (i1) + (s) + (.5,0)$);
\fill[red!30] (0,0) rectangle (i1);
\fill[blue!30] (0,10) rectangle (i1);
\fill[red!30] (s) rectangle ($(s) + (4,5)$);
\fill[blue!30] ($(s) + (4,10)$) rectangle (i2);
\foreach \y [count = \x] in {3,7,9,4,8,1,5,6,2}
	\fill (\x, \y) circle (1.2mm) node [above right] {\y};
\foreach \y [count = \x] in {3,4,1}
	\fill ($ (\x, \y) + (s) $) circle (1.2mm) node [above right] {\y};
\foreach \y [count = \x] in {7,9,8,5,6,2}
	\fill ($ (\x+3.5, \y) + (s) $) circle (1.2mm) node [above right] {\y};
\draw (i1) circle (2.5mm);
\draw (i2) circle (2.5mm);
\draw ($ .5*(s) + (4.5,5)$) node {\large $\longrightarrow$};
\end{tikzpicture}
\caption{Action of $\swl_5$ on the permutation $368472519$.}
\end{center}\end{figure}
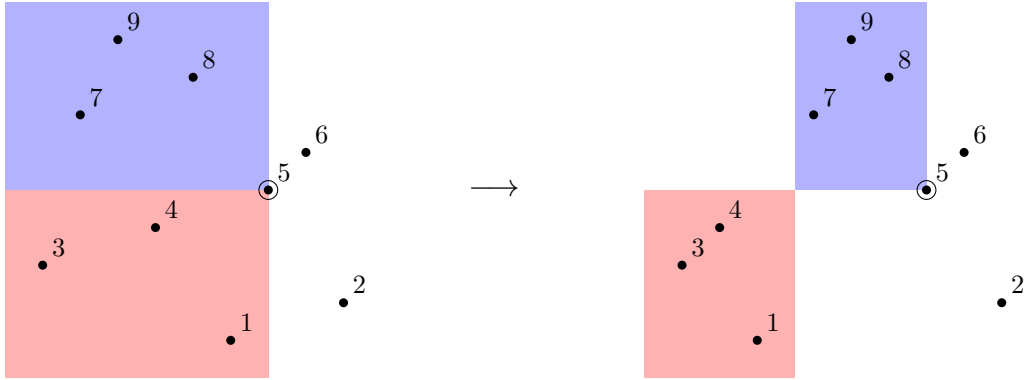

More formally, let $\pi\in S_n$ and let $(m,\pi_m)$ be the point of height $i$. Let $B$ and $T$ be the subsequences of $\pi_1\cdots\pi_{m-1}$ that consist of points $(j,\pi_j)$ with $\pi_j<i$ or $\pi_j>i$, respectively, and let $\pi_{\geq m}=\pi_m\pi_{m+1}\cdots\pi_n$. Then $\swl_i(\pi)$ is the concatenation $BT\pi_{\geq m}$. We define $\swr_i(\pi)$ similarly, except this time moving the ``southwest right.'' That is, $\swr_i(\pi)$ is the concatenation $TB\pi_{\geq m}$. Additionally, if $\beta=(m,\pi_m)$ is a point of $\pi$, we define $\swl_\beta(\pi)=\swr_{\pi_m}(\pi)$ and $\swr_\beta(\pi)=\swr_{\pi_m}(\pi)$.

This allows us to define the sliding operators $\swl$ and $\swr$.

\begin{defn}
Let $\pi\in S_n$. Then we define $\swl(\pi)=\swl_1\circ\cdots\circ\swl_n(\pi)$ and $\swr(\pi)=\swr_1\circ\cdots\circ\swr_n(\pi)$. 
\end{defn}

Defant \cite{hannaspaper} showed that $\swl:\Av_n(132)\to\Av_n(312)$ and $\swr:\Av_n(312)\to\Av_n(132)$ are inverse bijections. He additionally showed that $\swl$ gives an injection from uniquely sorted permutations (permutations in $S_n$ with a unique valid hook configuration and $\frac{n-1}2$ descents) avoiding 132 to uniquely sorted permutations avoiding 312, and that the composition $\widehat{\LL}_n\circ\swl$ gives a bijection from uniquely sorted permutations avoiding 132 to intervals in a subposet of $\M_{n-1}^T$. In the next section, we will generalize $\swl$ to an injection $\W_n:\VHC(\Av_n(132))\to\VHC(\Av_n(312))$. For convenience, we will only consider $\swl$ and $\swr$ to be defined on permutations avoiding 132 or 312, respectively.

The remainder of this section builds up results about the sliding operators that will be useful when studying $\W_n$. If $i<j$ are indices of a permutation, we say that $\pi_i$ and $\pi_j$ form an \emph{acclivity} if $\pi_i<\pi_j$ and a \emph{declivity} if $\pi_i>\pi_j$.

We describe when $\swr$ switches two points in the plot of $\pi$. In fact, $\swr$ changes every occurrence of the pattern 132 in $\pi$ into 312, and does no further transformations.

\begin{lemma}\label{lemma:swr properties}
Let $\pi$ avoid 312 and let $i<j$ be indices of $\pi$.
\begin{enumerate}[label=(\roman*)]
\item If $\pi_i>\pi_j$ then $\pi_i$ comes before $\pi_j$ in $\swr(\pi)$. That is, $\swr$ preserves declivities.
\item If $\pi_i<\pi_j$ then $\pi_i$ comes after $\pi_j$ in $\swr(\pi)$ if and only if there is some $k>j$ for which $(\pi_i,\pi_j,\pi_k)$ matches the pattern 132.
\end{enumerate}
\end{lemma}

\begin{proof}
(i) is clear because each $\swr_h$ preserves declivities, so $\swr$ must also preserve declivities. As a corollary, we note that it is impossible for $\pi_i$ and $\pi_j$ to be switched more than once by the sequence of maps $\swr_1\circ\cdots\circ\swr_n(\pi)$.

For (ii) we first assume there is such a $k$ and show that $\pi_i$ and $\pi_j$ are switched by $\swr$. Set $\tau=\swr_{\pi_k+1}\circ\cdots\circ\swr_n(\pi)$. No map in the sequence $\swr_{\pi_k+1}\circ\cdots\circ\swr_n$ will switch $\pi_j$ and $\pi_k$ because they form a declivity. Consider the relative positions of $\pi_i$ and $\pi_j$ in $\tau$. If $\pi_i$ comes after $\pi_j$ in $\tau$ then they form a declivity that will be preserved under all subsequent maps in $\swr_1\circ\cdots\circ\swr_{\pi_k}$. Alternatively, if $\pi_i$ comes before $\pi_j$ in $\tau$ then $\swr_{\pi_k}$ will switch $\pi_i$ and $\pi_j$ and all subsequent maps in $\swr_1\circ\cdots\circ\swr_{\pi_k-1}$ will again preserve the switch.

To finish (ii), we use induction on $\pi_j-\pi_i$ to show that if  $\pi_i$ and $\pi_j$ are switched by $\swr$ then there exists $k$ as described. If $\pi_j-\pi_i=1$ then $\pi_i$ and $\pi_j$ cannot be switched by any map in the composition sequence $\swr_1\circ\cdots\circ\swr_n$. For larger $\pi_j-\pi_i$, assume that $\swr_{\pi_h}$ switches $\pi_i$ and $\pi_j$. It immediately follows that $\pi_i<\pi_k<\pi_h$. We cannot have $h<i$ or else the declivity $(\pi_h,\pi_i)$ would remain in $\swr_{\pi_k+1}\circ\cdots\circ\swr_n(\pi)$. If $h>j$, we can take $k=h$. If $i<h<j$ then $\pi_h$ must be moved to the right of $\pi_j$ by some map in $\swr_{\pi_h+1}\circ\cdots\circ\swr_n$. By the inductive hypothesis, there is $k>j$ so that $(\pi_h, \pi_j, \pi_k)$ matches 132, and thus $(\pi_i,\pi_j,\pi_k)$ will also match 132.
\end{proof}

\begin{cor}\label{cor:swr properties}
Let $\pi\in S_n$ avoid 312 and let $i,j\in[n]$. Then $\swr(\pi)$ switches $\pi_i$ and $\pi_j$ if and only if $\swr_k(\pi)$ switches $\pi_i$ and $\pi_j$ for some $k\in[n]$. 
\end{cor}

\begin{proof}
Without loss of generality, let $i<j$. By Lemma \ref{lemma:swr properties}, $\swr(\pi)$ switches $\pi_i$ and $\pi_j$ if and only if there exists $k>j$ for which $(\pi_i,\pi_j,\pi_k)$ matches the pattern 132. However, the latter condition occurs by definition if and only if $\swr_{\pi_k}$ switches $\pi_i$ and $\pi_j$.
\end{proof}

Using similar techniques, or by noting that $\swl$ and $\swr$ are inverses, we can show Corollary \ref{cor:swr properties} for swl on 132-avoiding permutations as well. This gives rise to the following depictions of $\swl$ and $\swr$, first given by Defant in \cite{hannaspaper}.

\begin{fact}\label{fact:draw swl}
Let $\tau$ avoid 132. Then we can write
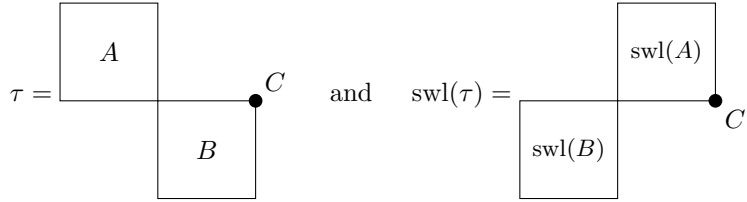
\[\tau=
\begin{tikzpicture}[baseline=(O.base), scale=1.3]
\node (O) at (0,0) {};
\draw (0,0) -- (0,1) -- (-1,1) -- (-1,0) -- (0,0);
\draw (-.5,.5) node {$A$};
\draw (0,0) -- (1,0) -- (1,-1) -- (0,-1) -- (0,0);
\draw (.5,-.5) node {$B$};
\fill (1,0) circle(.7mm) node[above right] {$C$};
\end{tikzpicture}
\quad\text{ and }\quad
\swl(\tau)=
\begin{tikzpicture}[baseline=(O.base), scale=1.3]
\node (O) at (0,0) {};
\draw (0,0) -- (0,-1) -- (-1,-1) -- (-1,0) -- (0,0);
\draw (-.5,-.5) node {\small$\swl(B)$};
\draw (0,0) -- (1,0) -- (1,1) -- (0,1) -- (0,0);
\draw (.5,.5) node {\small$\swl(A)$};
\fill (1,0) circle(.7mm) node[below right] {$C$};
\end{tikzpicture}
\]
where $A$ and $B$ are (potentially empty) regions of the plot of $\tau$ that are treated as plots of smaller 132-avoiding permutations, and where $C$ is a single point. Similarly, if $\pi$ avoids 312 then we can write
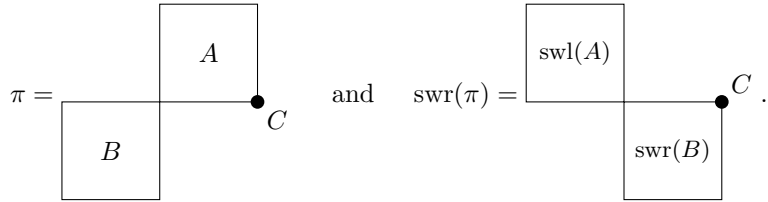
\[\pi=
\begin{tikzpicture}[baseline=(O.base), scale=1.3]
\node (O) at (0,0) {};
\draw (0,0) -- (0,-1) -- (-1,-1) -- (-1,0) -- (0,0);
\draw (-.5,-.5) node {$B$};
\draw (0,0) -- (1,0) -- (1,1) -- (0,1) -- (0,0);
\draw (.5,.5) node {$A$};
\fill (1,0) circle(.7mm) node[below right] {$C$};
\end{tikzpicture}
\quad\text{ and }\quad
\swr(\pi)=
\begin{tikzpicture}[baseline=(O.base), scale=1.3]
\node (O) at (0,0) {};
\draw (0,0) -- (0,1) -- (-1,1) -- (-1,0) -- (0,0);
\draw (-.5,.5) node {\small$\swl(A)$};
\draw (0,0) -- (1,0) -- (1,-1) -- (0,-1) -- (0,0);
\draw (.5,-.5) node {\small$\swr(B)$};
\fill (1,0) circle(.7mm) node[above right] {$C$};
\end{tikzpicture}
.\]
\end{fact}

\section{Extending $\swl$ to Valid Hook Configurations}

To construct a bijection between 132-avoiding valid hook configurations and Motzkin-Tamari intervals, we begin by constructing an injection $\W_n: \VHC(\Av_n(132))\to\VHC(\Av_n(312))$, which is modeled off of $\swl$. In the next section, we will define Defant's bijection $\widehat\LL_n:\VHC(\Av_n(312))\to\Int(\M_{n-1}^C)$ and show that the composition $\widehat{\LL}_n\circ\W_n$ has image exactly the Motzkin-Tamari intervals, thereby demonstrating a bijection from $\VHC(\Av_n(132))$ to $\VHC(\Av_n(312))$. This extends a bijection of Defant's in \cite{hannaspaper}, which mapped certain 132-avoiding permutations with unique valid hook configurations to intervals of a subposet of $\M_{n-1}^T$ by composing a variant of $\widehat\LL_n$ with $\swl$.

\begin{definition}
A point $P$ in the plot of $\pi$ is a \emph{left-to-right maximum} (resp.\ \emph{minimum}) if there are no points left of $P$ in the plot of $\pi$ that are strictly higher (resp.\ lower) than $P$.
\end{definition}

We note that northeast endpoints of hooks can never be left-to-right minima. When $\pi$ avoids 312, the northeast endpoints of a valid hook configuration must in fact be left-to-right maxima. This motivates the following map, which allows us to replace points with valid northeast endpoints in 312-avoiding permutations.

\begin{definition}
Let $\pi\in\Av(312)$. If $(i,\pi_i)$ is a point of $\pi$, we define $\nw(i,\pi_i)$ (which we pronounce ``northwest representative'') as the leftmost left-to-right maximum $(j,\pi_j)$ with $j\leq i$ and $\pi_j\geq \pi_i$.
\end{definition}

We need one more piece of terminology to define $\W_n$. Let $(\pi, V)\in\VHC(\Av_n(132))$, where we recall that $V$ is the set of northeast endpoints of some valid hook configuration on $\pi$. We say that the \emph{image} of a point $(i,\pi_i)$ of  $\pi$ under $\swl$ is the point of height $\pi_i$ in the image $\swl(\pi)$ and write this as $\swl(i,\pi_i)$, and define $\swr(i,\pi_i)$ analogously. It is natural to define image in this manner when discussing valid hook configurations, because this allows us to say $\swl$ and $\swr$ preserve southwest hook endpoints as in the following lemma.

\begin{lemma}\label{lemma:descent tops}
The maps $\swl$ and $\swr$ take descent tops to descent tops.
\end{lemma}

\begin{proof}
We show this for $\swl$, using induction and Fact \ref{fact:draw swl}. Assume $\tau\in S_n$ avoids 132. If $n=0$ or $n=1$ then $\tau$ has no descent tops. For larger $n$, we split $\tau$ into regions $A$ and $B$ and point $C$ as in Fact \ref{fact:draw swl}. By the inductive hypothesis, all descent tops of $\tau$ in $A$ or $B$ with descent bottom in the same region are taken to descent tops in $\swl(A)$ or $\swl(B)$. The only descent top of $\tau$ unaccounted for is the rightmost element of $A$ when $A$ is nonempty, as that point has descent bottom $C$. However, we see from Fact \ref{fact:draw swl} that the rightmost point of $A$ remains the rightmost point in $\swl(A)$, and that point must be a descent top in $\swl(\tau)$ with descent bottom either the leftmost element of $\swr(B)$, or $C$ if $B$ is empty.

The proof for $\swr$ is analogous.
\end{proof}

We now define $\W_n:\VHC(\Av_n(132))\to\VHC(\Av_n(312))$ as
\[\W_n(\pi, V)=(\swl(\pi), \nw(\swl(V))),\]
where $\nw(\swl(V))=\{\nw(\swl(A)):A\in V\}$. To show that $\W_n$ is well-defined and injective, we further study the map $\nw$. We can define an equivalence relation on points of the plot of $\pi$ based on their northwest representatives. The corresponding equivalence classes form horizontal stripes, as we will see in Proposition \ref{prop:stripes}.

\begin{definition}
The \emph{$\nw$ stripe} of a point $(i,\pi_i)$ is the equivalence class $\{(j,\pi_j):\nw(i,\pi_i)=\nw(j,\pi_j)\}$, denoted by $\s(i,\pi_i)$.
\end{definition}

\begin{prop}
Each $\nw$ stripe of a (312-avoiding) permutation $\pi$ is descending from left to right.
\end{prop}

\begin{proof}
Fix a left-to-right maximum $P$ of $\pi$ and consider the stripe $\s(P)$. Clearly $P\in\s(P)$, and $P$ is to the left and above every other point of $\s(P)$. If there were two points $Q,R\in\s(P)$ with $\pi$ ascending left from $Q$ to $R$ then $PQR$ would be an occurrence of the pattern 312 in $\pi$, which would be a contradiction. Hence the $\nw$ stripe $\s(P)$ is descending.
\end{proof}

\begin{prop}\label{prop:stripes}
If $P$ and $Q$ are left-to-right maxima in a 312-avoiding permutation $\pi$ and $P$ is below $Q$ then every point of $\s(P)$ is below every point of $\s(Q)$.
\end{prop}

\begin{proof}
Every point of $\s(P)$ is below $P$, so it suffices to show that $P$ is below any point of $\s(Q)$. Every point $R\in\s(Q)$ is to the right of $Q$ and therefore to the right of $P$. If $R$ were additionally below $P$ then $P$ would be a left-to-right maximum above and to the right of $R$ and left of $Q$, so $\nw(R)$ would not be $Q$. It follows that every such $R$ must be above $P$ and therefore above all of $\s(P)$.
\end{proof}

This allows us to define an ordering on the $\nw$ stripes of a 312-avoiding permutation by height.

\begin{theorem}\label{thm:stripe structure}
Let $\tau$ be a 132-avoiding permutation and let $\pi=\swl(\tau)$.
The least $\nw$ stripe of $\pi$ consists of only the images of left-to-right minima of $\tau$. Each subsequent $\nw$ stripe has only one element that is not the image of a left-to-right minimum of $\tau$, and that point is the rightmost point of the stripe.
\end{theorem}

\begin{proof}
We use induction. This clearly holds for the unique element of $S_1$. For larger permutations, we recall that Fact \ref{fact:draw swl} allows us to write
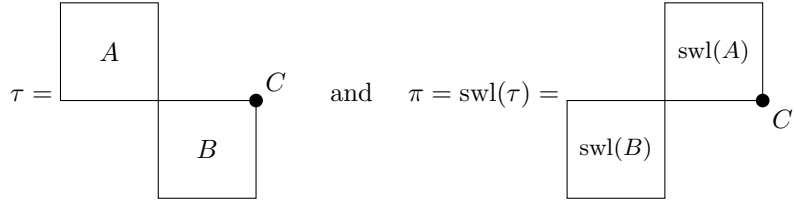
\[\tau=
\begin{tikzpicture}[baseline=(O.base), scale=1.3]
\node (O) at (0,0) {};
\draw (0,0) -- (0,1) -- (-1,1) -- (-1,0) -- (0,0);
\draw (-.5,.5) node {$A$};
\draw (0,0) -- (1,0) -- (1,-1) -- (0,-1) -- (0,0);
\draw (.5,-.5) node {$B$};
\fill (1,0) circle(.7mm) node[above right] {$C$};
\end{tikzpicture}
\quad\text{ and }\quad
\pi=\swl(\tau)=
\begin{tikzpicture}[baseline=(O.base), scale=1.3]
\node (O) at (0,0) {};
\draw (0,0) -- (0,-1) -- (-1,-1) -- (-1,0) -- (0,0);
\draw (-.5,-.5) node {\small$\swl(B)$};
\draw (0,0) -- (1,0) -- (1,1) -- (0,1) -- (0,0);
\draw (.5,.5) node {\small$\swl(A)$};
\fill (1,0) circle(.7mm) node[below right] {$C$};
\end{tikzpicture}
\]
where $C$ is a single point and $A$ and $B$ are (potentially empty) substrings of $\tau$.

Note that the left-to-right minima of $\tau$ are the left-to-right minima of $A$ and $B$, together with $C$ if $B$ is empty. Furthermore, the $\nw$ equivalence classes of $\pi$ are the $\nw$ equivalence classes of $A$ and $B$, with $C$ added to the least $\nw$ equivalence class of $A$ (if $A$ is empty, we treat it as having a single empty equivalence class).

If $B$ is empty, then the least equivalence class of $\tau$ is the least $\nw$ equivalence class of $A$, with $C$ added in; furthermore because $B$ is empty, $C$ is a left-to-right minimum of $\tau$. Any higher $\nw$ equivalence classes satisfy the desired property by the inductive hypothesis on $A$.

If $B$ is nonempty then every $\nw$ equivalence class of $\tau$ but the least class of $A$, which is also the class containing $C$, automatically satisfies the hypothesis by the inductive hypothesis on $A$ and $B$. The least element of this class is $C$, which is not the image of a left-to-right minimum in $\tau$ because $B$ is nonempty; all other elements of this class are images of left-to-right minima in $\tau$ by the inductive hypothesis on $A$.
\end{proof}

We can use Theorem \ref{thm:stripe structure} to show that $\W_n$ is well-defined and injective.

\begin{prop}
The map $\W_n$ is well-defined. That is, its image is always a valid hook configuration on the image permutation.
\end{prop}

\begin{proof}
Let $(\tau, V)\in\VHC(\Av_n(132))$ and let $(\pi, W)=\W_n(\tau, V)$. We will show that every hook in the valid hook configuration on $\tau$ maps to a hook with distinct endpoints and no points above it in $(\pi, W)$, which will show that $(\pi, W)$ is a valid hook configuration by Proposition \ref{prop:new VHC}.

Assume that there is a hook in $(\tau, V)$ with southwest and northeast endpoints $P$ and $Q$ respectively. By Lemma \ref{lemma:descent tops},  $\swl(P)$ is a descent top in $\pi$. Using reasoning analogous to the proof of Lemma \ref{lemma:swr properties} we can show that $\swl$ preserves acclivities, so $\swl(Q)$ is above and to the right of $\swl(P)$. Furthermore, $\nw(\swl(Q))$ must be strictly above and to the right of $\swl(P)$ because it must be above $\swl(Q)$, which is strictly above $\swl(P)$, and if it were to the left of $\swl(P)$, then $(\nw(\swl(Q)),\swl(P),\swl(Q))$ would match the pattern 312 in $\swl(\tau)$. Hence, the pair of points $\swl(P)$ and $\nw(\swl(Q))$ form a hook in $(\pi, W)$. There is no point in $\pi$ above this hook because $\nw(\swl(Q))$ is a left-to-right maximum of $\pi$.

To complete the proof, we show that the correspondence from descent tops in $\pi$ to points of $W$ is bijective, or that the northeast endpoints $Q\in V$ are mapped to distinct points $\nw(\swl(Q))\in W$. Note that if $Q$ is the northeast endpoint of a hook in $\tau$ then it cannot be a left-to-right minimum of $\tau$. By Theorem \ref{thm:stripe structure}, each point $Q\in V$ is taken to a different $\nw$ stripe of $\pi$ under $\swl$, and therefore maps to a different left-to-right maximum of $\pi$ under $\nw\circ\swl$. It follows that the hooks in the valid hook configuration $(\tau, V)$ map to hooks in $(\pi, W)$ that satisfy the conditions of Proposition \ref{prop:new VHC}, and therefore that $(\pi, W)$ is a valid hook configuration.
\end{proof}

\begin{prop}
The map $\W_n$ is injective.
\end{prop}

\begin{proof}
Given the image $(\swl(\tau),\nw(\swl(V)))$ of a valid hook configuration $(\tau,V)$, we can recover $\tau$ by applying $\swr$. We can recover $V$ by computing, for each point of $\nw(\swl(V))$, the unique preimage that is not a left-to-right minimum in $\tau$.
\end{proof}

In fact, if $(\pi, W)\in\VHC(\Av_n(312))$ we can always compute a preimage $(\tau, V)$ under $\W_n$ where $V$ is a set of points in the plot of $\tau$ that don't necessarily form a valid hook configuration. For a left-to-right maximum $A$ in the plot of $\pi$, define $\nw^{-1}(A)$ to be the rightmost element of $\s(A)$. Then we can define $\W^{-1}_n(\pi, W)=(\swr(\pi), \swr(\nw^{-1}(W)))$. It is clear that $\W_n^{-1}$ is a left inverse of $\W_n$; that is, $\W_n^{-1}\circ\W_n=\id_{\VHC(\Av_n(132))}$.

\section{The Image of $\W_n$}\label{section:define LL}

Having defined $\W_n$, we show that its composition with Defant's bijection $\widehat\LL_n:\VHC(\Av_n(312))\to\Int(\M_{n-1}^C)$ is surjective onto intervals of the subposet $\M_{n-1}^T$. To study the image of $\widehat\LL_n\circ\W_n$, we begin by defining the map $\widehat\LL_n$.

Let $(\pi, V)\in\VHC(\Av_n(312))$. We describe how to construct the Motzkin interval $(P, Q)=\widehat\LL_n(\pi, V)$ with $P\leq_C Q$. Let $\R_0,\ldots,\R_\ell$ be the left-to-right maxima of $\pi$ listed from right to left. That is, $\R_\ell=(1,\pi_1)$; moreover, because $\pi$ has a valid hook configuration, we have $\pi_n=n$ and $\R_0=(n,n)$. For convenience let $\R_{\ell+1}=(0,0)$. For $i\in[\ell]$, let $\Gamma_i$ be the set of points strictly right of $\R_i$ and strictly left of $\R_{i-1}$, and let $\Gamma'_i$ be the set of points strictly above $\R_{i+1}$ and strictly below $\R_i$. For each $i\in[\ell]$, set $\gamma_i=|\Gamma_i|$ and $\gamma'_i=|\Gamma'_i|$, and let $X_i$ be $U$ if $\R_{i-1}\in V$ and $E$ otherwise. The definition of $\widehat\LL_n$ in \cite{Defant} states that $\widehat\LL_n(\pi, V)=(P, Q)$ where $P=X_1D^{\gamma_1}\cdots X_\ell D^{\gamma_\ell}$ and $Q=X_1D^{\gamma'_1}\cdots X_\ell D^{\gamma'_\ell}$. We provide an example of this construction in Figure \ref{fig:LL} below.
\begin{figure}[h]\begin{center}
	\begin{tikzpicture}[scale=.5, baseline=(O)]
	\foreach \y [count=\x] in {3, 2, 4, 1, 5, 6}
		\fill (\x, \y) circle (1.3mm);
	\draw (1,3) -- (1,4) -- (3,4) -- (3, 6) -- (6,6);
	\coordinate (O) at (1,3);
	\node at (3.5,-.3) {$\pi=324156$};
	\end{tikzpicture}
	\quad$\implies$\quad
	\begin{tikzpicture}[scale=.5, baseline=(R3)]
	\foreach \y [count=\x] in {3, 2, 4, 1, 5, 6}
		\fill (\x, \y) circle (1.3mm);
	\coordinate (R0) at (6,6); \node at (R0) [above left] {$\R_0$};
	\coordinate (R1) at (5,5); \node at (R1) [above left] {$\R_1$};
	\coordinate (R2) at (3,4); \node at (R2) [above left] {$\R_2$};
	\coordinate (R3) at (1,3); \node at (R3) [above left] {$\R_3$};
	\coordinate (R4) at (0,0); \node at (R4) [below] {$\R_4=(0,0)\,\,\,\,\,\,\,$};
	\coordinate (top) at (0,7.2); \coordinate (bottom) at (0,-1.2);
	\coordinate (left) at (-1.2,0); \coordinate (right) at (7.2,0);
	\fill[gray] (R4) circle (1.3mm);
	\foreach \pt in {R0, R1, R2, R3}
		\draw[dashed] (\pt|-top) -- (\pt|-bottom);
	\foreach \p/\q [count=\i] in {R0/R1,R1/R2,R2/R3}
		\node at ($ (\p|-top)!.5!(\q|-top) $) [above] {$\Gamma_\i$};
	\foreach \pt in {R1, R2, R3, R4}
		\draw[dashed] (\pt-|left) -- (\pt-|right);
	\foreach \p/\q [count=\i] in {R1/R2,R2/R3,R3/R4}
		\node at ($ (\p-|left)!.5!(\q-|left) $) [left] {$\Gamma'_\i$};
	\end{tikzpicture}
	\quad
	$\begin{array}{cc}
	\implies&\begin{array}{c}(\gamma_1,\gamma_2,\gamma_3)=(0,1,1)
	\\(\gamma'_1,\gamma'_2,\gamma'_3)=(0,0,2)\end{array}
	\\&\rotatebox{-90}{\,\,$\implies\,\,\,\,\,\,$}
	\\&P=UEDUD
	\\&Q=UEUDD
	\end{array}$
\caption{Constructing $\widehat\LL_n$ for the valid hook configuration on $\pi=324156$ shown at left.}
\label{fig:LL}
\end{center}\end{figure}
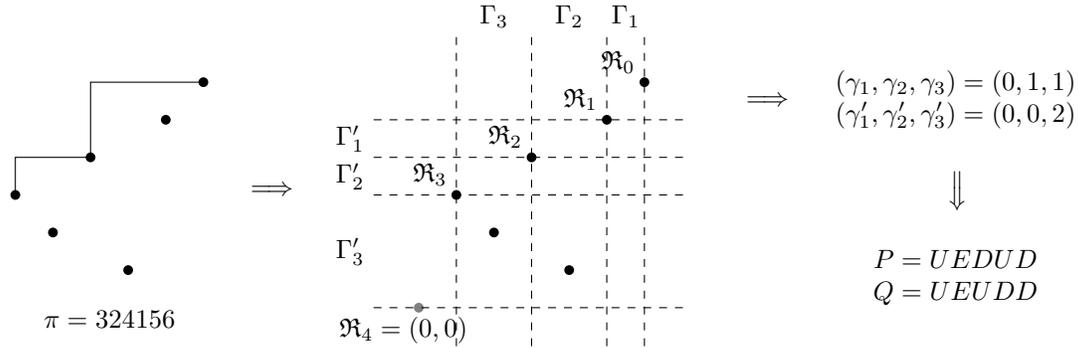

Set $(\tau, W)=\W_n^{-1}(\pi, V)=(\swr(\pi), \swr(\nw^{-1}(V)))$. We will show that if $P\nleq_T Q$ then $W$ is an invalid hook configuration in $\tau$, and that if $P\leq_TQ$ then the preimage of each hook in the valid hook configuration $(\pi, V)$ under $\W_n$ is a distinct hook on $\tau$ satisfying the conditions of Proposition \ref{prop:new VHC}.

\begin{prop}\label{prop:tamari properties}
As before, denote our Motzkin paths $P\leq_CQ$ by $P=X_1D^{\gamma_1}\cdots X_\ell D^{\gamma_\ell}$ and $Q=X_1D^{\gamma'_1}\cdots X_\ell D^{\gamma'_\ell}$.
Assume $P\nleq_TQ$ and fix an index $i\in[\ell]$ for which $\lng_i(P)>\lng_i(Q)$. Then there is $k\geq 0$ such that
\begin{enumerate}[label=(\roman*)]
\item $\gamma_i+\cdots+\gamma_{i+k}<\gamma'_i+\cdots+\gamma'_{i+k}$, and
\item For all $j$ with $0\leq j\leq k$, we have $\gamma_i+\cdots+\gamma_{i+j}<\delta(X_i)+\cdots+\delta(X_{i+j})$,
\end{enumerate}
where $\delta(s)$ represents the $y$-displacement of a step $s$ of a Motzkin path as in Definition \ref{def:y-displacement}.
\end{prop}

\begin{proof}
Let the subpath of $Q$ of length $\lng_i(Q)$ starting at $X_i$ be $Q^{(i)}=X_iD^{\gamma'_i}X_{i+1}D^{\gamma'_{i+1}}\cdots X_{i+k}D^{\gamma'_{i+k}-\eta}$ for appropriate $k,\eta\geq 0$, and let $\alpha=\delta(X_i)+\cdots+\delta(X_{i+k})$ be the number of $U$'s in $Q^{(i)}$. Because $Q^{(i)}$ is Motzkin,  $\alpha$ is also the number of $D$'s in $Q^{(i)}$, so $\alpha=\gamma'_i+\cdots+\gamma'_{i+k}-\eta\leq\gamma'_i+\cdots+\gamma'_{i+k}$. Moreover, $|Q^{(i)}|=k+1+\alpha$ because $k+1$ is the number of $U$ or $E$ steps in $Q^{(i)}$.

Now, choose $j$ with $0\leq j\leq k$ and set $\beta=\delta(X_i)+\cdots+\delta(X_{i+j})$. If $\gamma_i+\cdots+\gamma_{i+j}\geq\beta$ then the subpath $P'=X_iD^{\gamma_i}\cdots X_{i+j}D^{\gamma_{i+j}}$ would end at or below the horizontal and therefore would have a Motzkin prefix $P^{(i)}$. The number of $D$'s in $P'$ would be at most $\beta$, which is the number of $D$'s in $P$, so $|P^{(i)}|\leq|P'|\leq j+1+\beta$. However, $j\leq k$ and $\beta\leq\alpha$, so $|P^{(i)}|\leq|Q^{(i)}|$, which would imply the contradiction $\lng_i(P)\leq\lng_i(Q)$. Therefore, $\gamma_i+\cdots+\gamma_{i+j}<\beta=\delta(X_i)+\cdots+\delta(X_{i+j})$ for all $0\leq j\leq k$, proving (ii).

In the special case $j=k$, this yields $\gamma_i+\cdots+\gamma_{i+k}<\alpha\leq\gamma'_i+\cdots+\gamma'_{i+k}$,  proving (i).
\end{proof}

\begin{theorem}\label{thm:images are tamari}
Let $(P,Q)=\widehat\LL_n(\pi,V)$ as before. If $P\nleq_TQ$ then $\W_n^{-1}(\pi,V)=(\tau,W)$ is not a valid hook configuration.
\end{theorem}

\begin{proof}
We can choose $i$ and $k$ to satisfy Proposition \ref{prop:tamari properties}. For $S\subset[\ell]$, let $\Gamma_S=\bigcup_{m\in S}\Gamma_m$ and $\Gamma'_S=\bigcup_{m\in S}\Gamma'_m$. Define regions $A$, $B$, and $C$ by setting $C=\Gamma_{[i,i+k]}\cap\Gamma'_{[i,i+k]}$ and letting $A=\Gamma_{[i,i+k]}\setminus C$ and $B=\Gamma'_{[i,i+k]}\setminus C$, as depicted in Figure \ref{fig:images are tamari}. Regions marked in red are those where no point can lie.

\begin{figure}[h]\begin{center}
\begin{tikzpicture}[scale=1]
	\coordinate (R8) at (0,0);
	\coordinate (R7) at ($ (R8) + (1,1) $);
	\coordinate (R6) at ($ (R7) + (1,1) $);
	\coordinate (R5) at ($ (R6) + (1.5,1.5) $);
	\coordinate (R4) at ($ (R5) + (1,1) $);
	\coordinate (R3) at ($ (R4) + (1.5,1.5) $);
	\coordinate (R2) at ($ (R3) + (1,1) $);
	\coordinate (R1) at ($ (R2) + (1,1) $);
	\coordinate (left) at (-1.2,0);
	\coordinate(right) at ($ (R1) + (1.2,0) $);
	\coordinate (top) at ($ (R1) + (0,1.2) $);
	\coordinate (bottom) at (0,-1.2);
	
	\fill[red!30] (left |- top) -- (R1 |- top) -- (R1) -- (R1 |- R2) -- (R2) -- (R2 |- R3) -- (R3) -- (R4) -- (R4 |- R5) -- (R5) -- (R6) -- (R6 |- R7) -- (R7) -- (R7 |- R8) -- (R8) -- (left |- R8);
	\fill[green!30] (R1 |- R8) -- (R1 |- R2) -- (R2) -- (R2 |- R3) -- (R3) -- (R4) -- (R4 |- R5) -- (R5) -- (R6) -- (R6 |- R7) -- (R7) -- (R7 |- R8) -- (R8);
	\fill[blue!22] (R1 |- R2) -- (right |- R2) -- (right |- R8) -- (R1 |- R8);
	\fill[yellow!30] (R7 |- R8) -- (R7 |- bottom) -- (R1 |- bottom) -- (R1 |- R8);
	\draw ($ (R6-|R4)!.5!(R5-|R3) $) node [scale=1.5, green!30!black] {$C$};
	\draw ($ (R6-|R1)!.5!(R5-|right) $) node [scale=1.5, blue!40!black] {$B$};
	\draw ($ (R8-|R4)!.5!(bottom-|R3) $) node [scale=1.5,yellow!30!black] {$A$};
	
	\fill ($ (R3)!.5!(R4) $) circle (.32mm) +(.12,.12) circle(.32mm) +(-.12,-.12) circle(.32mm);
	\fill ($ (R5)!.5!(R6) $) circle (.32mm) +(.12,.12) circle(.32mm) +(-.12,-.12) circle(.32mm);
	\draw[dashed] (top -| R8) -- (bottom -| R8);
	\draw[dashed] (top -| R7) -- (bottom -| R7);
	\draw[dashed] (top -| R6) -- (bottom -| R6);
	\draw[dashed] (top -| R5) -- (bottom -| R5);
	\draw[dashed] (top -| R4) -- (bottom -| R4);
	\draw[dashed] (top -| R3) -- (bottom -| R3);
	\draw[dashed] (top -| R2) -- (bottom -| R2);
	\draw[dashed] (top -| R1) -- (bottom -| R1);
	\node (g) at ($ (top -| R1)!.5!(top -| R2) $)[above] {$\Gamma_i$};
	\node (g1) at ($ (top -| R2)!.5!(top -| R3) $)[above] {$\Gamma_{i+1}$};
	\node (g2) at ($ (top -| R3)!.5!(top -| R4) $)[above] {$\cdots$};
	\node (g3) at ($ (top -| R4)!.5!(top -| R5) $)[above] {$\Gamma_{i+j+1}$};
	\node (g4) at ($ (top -| R5)!.5!(top -| R6) $)[above] {$\cdots$};
	\node (g5) at ($ (top -| R6)!.5!(top -| R7) $)[above] {$\Gamma_{i+k}$};
	\node (g6) at ($ (top -| R7)!.5!(top -| R8) $)[above] {$\Gamma_{i+k+1}$};
	\draw[dashed] (left |- R8) -- (right |- R8);
	\draw[dashed] (left |- R7) -- (right |- R7);
	\draw[dashed] (left |- R6) -- (right |- R6);
	\draw[dashed] (left |- R5) -- (right |- R5);
	\draw[dashed] (left |- R4) -- (right |- R4);
	\draw[dashed] (left |- R3) -- (right |- R3);
	\draw[dashed] (left |- R2) -- (right |- R2);
	\draw[dashed] (left |- R1) -- (right |- R1);
	\node (g') at ($ (left |- R1)!.5!(left |- R2) $)[left] {$\Gamma'_{i-1}$};
	\node (g'0) at ($ (left |- R2)!.5!(left |- R3) $)[left] {$\Gamma'_i$};
	\node (g'1) at ($ (left |- R3)!.5!(left |- R4) $)[left] {$\vdots\,\,\,\,$};
	\node (g'2) at ($ (left |- R4)!.5!(left |- R5) $)[left] {$\Gamma'_{i+j}$};
	\node (g'3) at ($ (left |- R5)!.5!(left |- R6) $)[left] {$\vdots\,\,\,\,$};
	\node (g'4) at ($ (left |- R6)!.5!(left |- R7) $)[left] {$\Gamma'_{i+k-1}$};
	\node (g'5) at ($ (left |- R7)!.5!(left |- R8) $)[left] {$\Gamma'_{i+k}$};
	\fill (R1) circle(.7mm) node[above left] {$\R_{i-1}$};
	\fill (R2) circle(.7mm) node[above left] {$\R_{i}$};
	\fill (R3) circle(.7mm) node[above left] {$\R_{i+1}$};
	\fill (R4) circle(.7mm) node[above left] {$\R_{i+j}$};
	\fill (R5) circle(.7mm) node[above left] {$\R_{i+j+1}$};
	\fill (R6) circle(.7mm) node[above left] {$\R_{i+k-1}$};
	\fill (R7) circle(.7mm) node[above left] {$\R_{i+k}$};
	\fill (R8) circle(.7mm) node[above left] {$\R_{i+k+1}$};
	\fill ($ (R5 -| R1)!.4!(R4 -| R1) +(.5,0)$) circle (.7mm) node[above right] {$\beta$};
\end{tikzpicture}
\caption{The plot of $\pi$.}
\label{fig:images are tamari}
\end{center}\end{figure}
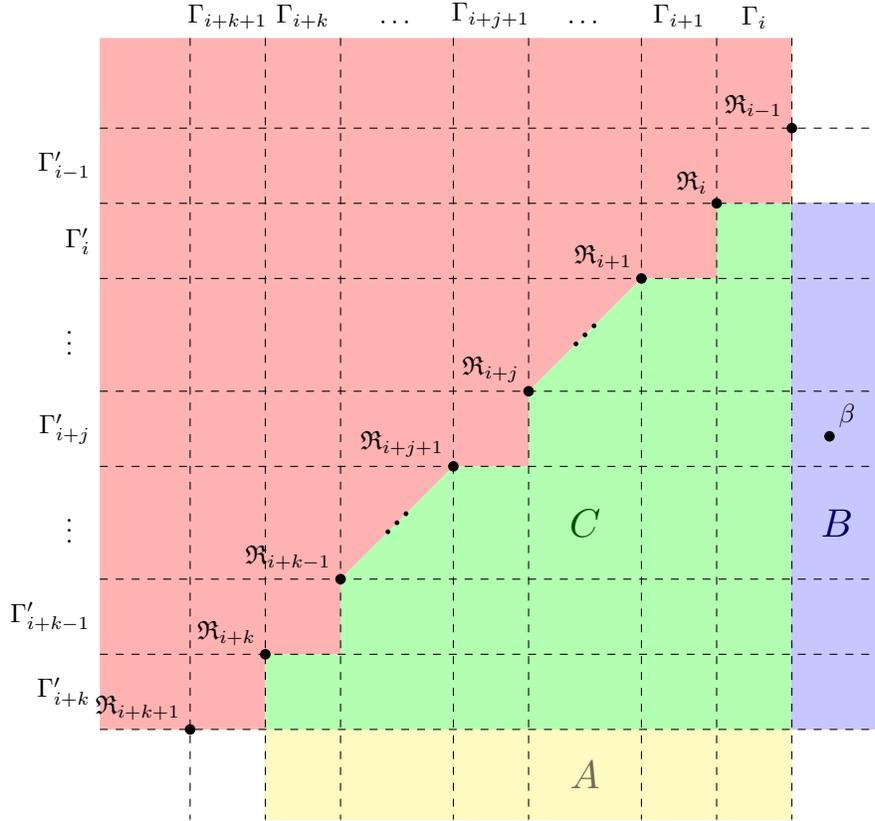

We have $|A|+|C|=\gamma_i+\cdots+\gamma_{i+k}<\gamma'_i+\cdots+\gamma'_{i+k}=|B|+|C|$ by Proposition \ref{prop:tamari properties}, so $B$ must contain at least one point, $\beta$. Choose $j$ so that $\beta\in\Gamma'_{i+j}$, as pictured in Figure \ref{fig:images are tamari}. Let $S=V\cap\{\R_{i-1},\ldots,\R_{i+j-1}\}$ be the set of northeast hook endpoints among those left-to-right maxima and choose any hook endpoint $R\in S$. If $\nw^{-1}(R)\neq R$ then $\nw^{-1}(R)$ must still be to the left of $\beta$ or else $(R,\beta,\nw^{-1}(R))$ would form a 312 in $\pi$. Let $T$ be the set of descent tops that lie above $\beta$ and strictly left of $\R_{i-1}$. Because $\nw^{-1}(R)$ is above and to the left of $\beta$, we have that $\swr_\beta(\pi)$ shifts $\nw^{-1}(R)$ to the left of all descent tops in $\pi$ except for some of the descent tops in $T$. Furthermore, every descent top in $T$ has corresponding descent bottom in $\Gamma_{[i,i+j]}$, so $|T|\leq\gamma_i+\cdots+\gamma_{i+j}$.

We will show that $W$ is not a valid hook configuration on $\tau$ by showing that  there are fewer than $|S|$ descent tops of $\tau$ that are southwest of any northeast hook endpoint in $\swr(\nw^{-1}(S))\subseteq W$. Let $R\in S$. If $R'$ is a descent top in $\pi$ southwest of $\nw^{-1}(R)$ and $\swr_\beta(R')$ is southeast of $\swr_\beta(\nw^{-1}(R))$, then by Corollary \ref{cor:swr properties}, $\swr(R')$ will be southeast of $\swr(\nw^{-1}(R))$. Furthermore, $\swr$ preserves declivities, so if $R'$ is a descent top in $\pi$ southeast of $\nw^{-1}(R)$ then $\swr(R')$ will be southeast of $\swr(\nw^{-1}(R))$. It follows that the only descent tops of $\tau$ southwest of any point in $\swr(\nw^{-1}(S))$ are elements of $\swr(T)$. However, by Proposition \ref{prop:tamari properties}, we have $|T|\leq\gamma_i+\cdots+\gamma_{i+j}<\delta(X_i)+\cdots+\delta(X_{i+j})=|S|$. Hence, there are fewer than $|S|$ descent tops of $\tau$ southwest of any point in $\swr(\nw^{-1}(S))$, and $W\supseteq S$ is an invalid hook configuration on $\tau$.
\end{proof}

To prove the opposite direction, we start with the valid hook configuration on $(\pi, V)$ and use it to prove that $(\tau, W)$ represents a valid hook configuration as well.

\begin{prop}\label{prop:length of long}
Assume that the valid hook configuration on $(\pi, V)$ contains a hook with southwest endpoint $A$ and northeast endpoint $B=\R_{i-1}$. Then $\lng_i(P)$ is equal to the horizontal distance between $A$ and $B$ in the plot of $\pi$, where $P$ is the first component of $\widehat\LL_n(\pi, V)$.
\end{prop}

\begin{proof}
We can construct $P$ by listing the points of $\pi$ from right to left, and mapping northeast hook endpoints to $U$, descent bottoms to $D$, and all other points to $E$. In this representation, $B=\R_{i-1}$ is the point corresponding to $X_i$. Let $C$ be the descent bottom corresponding to $A$. For every other hook, either both or neither of its northeast hook endpoint and descent bottom lie between $C$ and $B$ horizontally. It follows that, listing the points of $\pi$ from right to left starting at $B$ and ending at some point before $C$, there will be more northeast hook endpoints listed than descent bottoms, so the corresponding substring of $P$ is not Motzkin. However, if we list the points of $\pi$ from right to left starting at $B$ and ending at $C$, we will have listed exactly the same number of northeast hook endpoints and descent bottoms, so the corresponding substring of $P$ is a Motzkin subpath. It follows that $\lng_i(P)$ is the number of points lying horizontally between $B$ and $C$, inclusive, which is exactly the horizontal distance between $A$ and $B$ in the plot of $\pi$.
\end{proof}

Our end goal is to show that a hook in the valid hook configuration $(\pi, V)$ maps to a hook in the preimage $(\tau,W)$, which then allows us to apply Proposition \ref{prop:new VHC}. To do this, we study the region of points which swap the preimages of $A$ and $B$ in $\tau$.

\begin{definition}\label{def:pivot point}Let $\H$ be a hook in the plot of a permutation $\pi$ with southwest and northeast endpoints $A$ and $B$, respectively. A \emph{pivot point} of $\H$ is a point $\rho$ in the plot of $\pi$ such $(A,\nw^{-1}(B), \rho)$ matches the pattern 132.\end{definition}
It is clear from Definition \ref{def:pivot point} that the existence of no pivot point of $\H$ is equivalent to the preimage of $\H$ being a hook in $\tau$, by Corollary \ref{cor:swr properties}.

\begin{prop}\label{prop:no pivot points}
Let $(P,Q)=\widehat\LL_n(\pi,V)$ and assume $P\leq_T Q$. Let $\H$ be a hook in the valid hook configuration on $(\pi, V)$ with southwest endpoint $A$, northeast endpoint $B$, and descent bottom $C$. Then $\pi$ contains no pivot points that lie between $B$ and $C$ vertically and to the right of $\nw^{-1}(B)$.
\end{prop}

\begin{proof}
Assume that some hooks of the valid hook configuration had pivot points as described. Let $\rho$ be the highest such pivot point. We say hook $\H_1$ \emph{shelters} hook $\H_2$ if both endpoints of $\H_2$ lie under $\H_1$. If two hooks both have $\rho$ as a pivot point, it is clear that one hook must shelter the other, and we can pick $\H$ to be the most sheltered hook for which $\rho$ is a pivot point. Let $A$, $B$, and $C$ be the hook endpoints and descent bottom of $\H$ as described. Choose indices $i,j,k\in[\ell]$ such that $B=\R_{i-1}$, $A\in\{\R_j\}\cup \Gamma'_j$, and $\rho\in\Gamma'_k$, as shown in Figure \ref{fig:pivot point}.
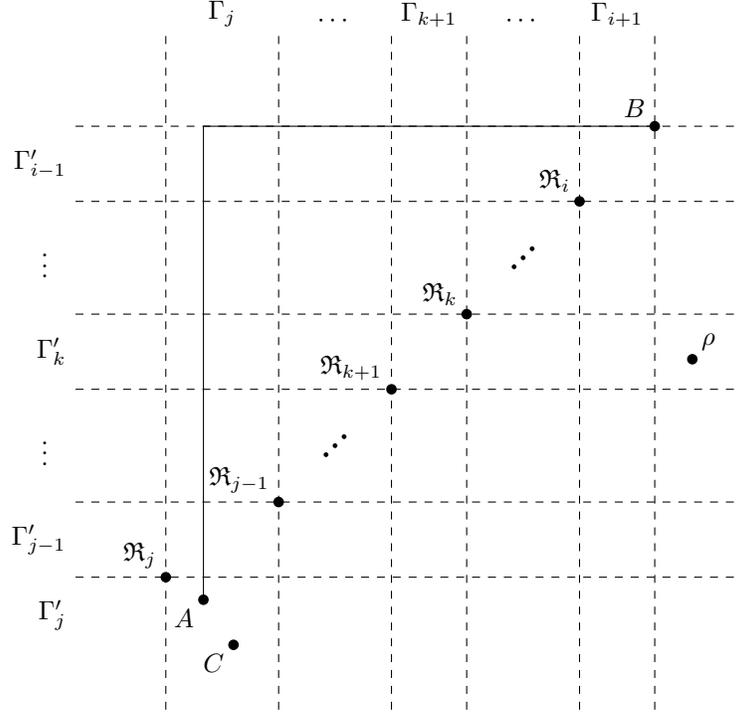
\begin{figure}[h]\begin{center}
\begin{tikzpicture}[scale=1]
	\coordinate (R8) at (0,0);
	\coordinate (R7) at ($ (R8) + (1,1) $);
	\coordinate (A) at ($ (R7) + (.5,-.3) $);
	\coordinate (C) at ($ (A) + (.4,-.6) $);
	\coordinate (R6) at ($ (R7) + (1.5,1) $);
	\coordinate (R5) at ($ (R6) + (1.5,1.5) $);
	\coordinate (R4) at ($ (R5) + (1,1) $);
	\coordinate (R3) at ($ (R4) + (1.5,1.5) $);
	\coordinate (R2) at ($ (R3) + (1,1) $);
	\coordinate (B) at (R2);
	\coordinate (left) at (-.2,0);
	\coordinate(right) at ($ (R2) + (1.2,0) $);
	\coordinate (top) at ($ (R2) + (0,1.2) $);
	\coordinate (bottom) at (0,-.8);
	
	\draw (A) -- (A|-B) -- (B);
		
	\fill ($ (R3)!.5!(R4) $) circle (.32mm) +(.12,.12) circle(.32mm) +(-.12,-.12) circle(.32mm);
	\fill ($ (R5)!.5!(R6) $) circle (.32mm) +(.12,.12) circle(.32mm) +(-.12,-.12) circle(.32mm);
	\draw[dashed] (top -| R7) -- (bottom -| R7);
	\draw[dashed] (top -| R6) -- (bottom -| R6);
	\draw[dashed] (top -| R5) -- (bottom -| R5);
	\draw[dashed] (top -| R4) -- (bottom -| R4);
	\draw[dashed] (top -| R3) -- (bottom -| R3);
	\draw[dashed] (top -| R2) -- (bottom -| R2);
	\node (g1) at ($ (top -| R2)!.5!(top -| R3) $)[above] {$\Gamma_{i+1}$};
	\node (g2) at ($ (top -| R3)!.5!(top -| R4) $)[above] {$\cdots$};
	\node (g3) at ($ (top -| R4)!.5!(top -| R5) $)[above] {$\Gamma_{k+1}$};
	\node (g4) at ($ (top -| R5)!.5!(top -| R6) $)[above] {$\cdots$};
	\node (g5) at ($ (top -| R6)!.5!(top -| R7) $)[above] {$\Gamma_{j}$};
	\draw[dashed] (left |- R7) -- (right |- R7);
	\draw[dashed] (left |- R6) -- (right |- R6);
	\draw[dashed] (left |- R5) -- (right |- R5);
	\draw[dashed] (left |- R4) -- (right |- R4);
	\draw[dashed] (left |- R3) -- (right |- R3);
	\draw[dashed] (left |- R2) -- (right |- R2);
	\node (g'0) at ($ (left |- R2)!.5!(left |- R3) $)[left] {$\Gamma'_{i-1}$};
	\node (g'1) at ($ (left |- R3)!.5!(left |- R4) $)[left] {$\vdots\,\,\,\,$};
	\node (g'2) at ($ (left |- R4)!.5!(left |- R5) $)[left] {$\Gamma'_{k}$};
	\node (g'3) at ($ (left |- R5)!.5!(left |- R6) $)[left] {$\vdots\,\,\,\,$};
	\node (g'4) at ($ (left |- R6)!.5!(left |- R7) $)[left] {$\Gamma'_{j-1}$};
	\node (g'5) at ($ (left |- R7)!.5!(left |- R8) $)[left] {$\Gamma'_{j}$};
	\fill (A) circle(.7mm) node[below left] {$A$};
	\fill (C) circle(.7mm) node[below left] {$C$};
	\fill (R2) circle(.7mm) node[above left] {$B$};
	\fill (R3) circle(.7mm) node[above left] {$\R_{i}$};
	\fill (R4) circle(.7mm) node[above left] {$\R_{k}$};
	\fill (R5) circle(.7mm) node[above left] {$\R_{k+1}$};
	\fill (R6) circle(.7mm) node[above left] {$\R_{j-1}$};
	\fill (R7) circle(.7mm) node[above left] {$\R_{j}$};
	\fill ($ (R5 -| R2)!.4!(R4 -| R2) +(.5,0)$) circle (.7mm) node[above right] {$\rho$};
\end{tikzpicture}
\caption{A hook in the plot of $\pi$ with a pivot point.}
\label{fig:pivot point}
\end{center}\end{figure}
Note that $k\geq i$ because there are no points in $\Gamma'_{i-1}$ to the right of $\nw^{-1}(B)$ by definition of $\nw^{-1}$. Furthermore, $k<j$ because if $\rho\in\Gamma'_j$ were above $A$ then $(\R_j,A,\rho)$ would match the pattern 312.

We can construct $Q$ in a manner similar to the construction of $P$ in Proposition \ref{prop:length of long}. Instead of listing the points of $\pi$ from right to left as with $P$, we list them in an almost-descending order: $\R_0$, points of $\Gamma'_1$ from highest to lowest, $\R_1$, points of $\Gamma'_2$ from highest to lowest, $\R_2$, and so on. Then, we convert this sequence of points into the Motzkin path $Q$ by mapping northeast hook endpoints to $U$, descent bottoms to $D$, and all other points to $E$ as before. Consider the subsequence of $Q$ that begins with the $U$ corresponding to $B$ and ends with the $D$ corresponding to $\rho$. Because $\rho$ is the highest pivot point of $\H$, all points in the almost-descending order between $B$ and $\rho$ are sheltered by $\H$. Furthermore, we claim that a hook $\H'$ has northeast endpoint $B'$ between $B$ and $\rho$ in this order if and only if its descent bottom $C'$ also lies between $B$ and $\rho$. Clearly if $C'$ were between $B$ and $\rho$ then $B'$ would be a left-to-right maximum $\R_a$ with $a>k$, and because $\H$ shelters $C'$, $\H$ and $\H'$ would intersect if $a\geq i-1$. Conversely, if $B'$ were a left-to-right maximum strictly between $\R_k$ and $\R_{i-1}$ then the southwest endpoint $A'$ of $\H'$ would also be above $\rho$, or else $\H'$ would be a hook more sheltered than $\H$ for which $\rho$ was a pivot point. Then $C'$ must also be above $\rho$, or else $(A', C', \rho)$ would match the pattern 312.

It follows that the subsequence of $Q$ corresponding to the points from $B$ to $\rho$ in almost-descending order contains as many $U$'s as $D$'s: one $U$ and one $D$ for each hook sheltered by $\H$ above $\rho$ and one $U$ from $B$ and one $D$ from $\rho$. However, every point of this subsequence except $\rho$ is also horizontally between $B$ and $C$, inclusive, and by Proposition \ref{prop:length of long} also contributes to $\lng_i(P)$. Moreover, the points $C$ and $\R_{j-1}$ both contribute to $\lng_i(P)$ but not to $\lng_i(Q)$. It follows that $\lng_i(Q)-1\leq\lng_i(P)-2$, which contradicts the assumption that $P\leq_T Q$. Hence, our initial assumption that there was a pivot point is false.
\end{proof}

\begin{cor}
The composition $\widehat{\LL}_n\circ \W_n$ is surjective onto the Tamari intervals $\Int(\M_{n-1}^T)$.
\end{cor}

\begin{proof}
The contrapositive of Theorem \ref{thm:images are tamari} shows that the composition does in fact take every element of $\VHC(\Av_n(132))$ to a Motzkin-Tamari interval. To show that any Motzkin-Tamari interval $(P,Q)\in\Int(\M_{n-1}^T)$ has a preimage under the composition, we let $(\pi, V)\in\VHC(\Av_n(312))$ be the preimage of $(P,Q)$ under $\widehat{\LL}_n$ and take $(\tau, W)=\W_n^{-1}(\pi,V)$. By Proposition \ref{prop:no pivot points} and Corollary \ref{cor:swr properties}, any hook $\H$ as in Proposition \ref{prop:no pivot points} is mapped to a hook in $(\tau, W)$, as $\swr$ will preserve the relative orders of $A$ and $\swr(B)$. We can then apply Proposition \ref{prop:new VHC} to show that $(\tau,W)$ is a valid hook configuration.
\end{proof}

\section{Future Work}\label{section:future work}

\subsection{Reduced Valid Hook Configurations}

During initial attempts to prove Defant's conjecture about valid hook configurations on $312$-avoiding permutations (Corollary \ref{cor:count with w}), the author was led to consider the enumeration of valid hook configurations according to the number of hooks. This has already been studied to some extent in the papers \cite{hannaspaper, DefantEngenMiller, hanna}, which investigate uniquely sorted permutations. Indeed, uniquely sorted permutations in $S_n$ are essentially the same as valid hook configurations on permutations in $S_n$ that have exactly $\frac{n-1}{2}$ hooks. We wish to extend the exploration of uniquely sorted permutations with what we call ``reduced" valid hook configurations.

\begin{definition}
A valid hook configuration is \emph{reduced} if every point in the plot of the permutation is
either a hook endpoint (both southwest and northeast endpoints count) or a descent bottom. Given $S\subseteq S_n$, we write $\VHC_k(S)$ for the set of valid hook configurations on permutations in $S$ with exactly $k$ hooks. Let $\RedVHC(S)$ be the set of reduced valid hook configurations on permutations in $S$, and let $\RedVHC_k(S)=\RedVHC(S)\cap\VHC_k(S)$.  
\end{definition}

Suppose $(\pi,V)\in\VHC(\Av_n(312))$, and let $S$ be the set of indices $i\in[n]$ such that $(i,\pi_i)$ is a hook endpoint or descent bottom of $(\pi,V)$. Let $(\pi,V)\big\vert_S$ be the valid hook configuration obtained by removing the points of the plot of $\pi$ of the form $(j,\pi_j)$ with $j\in[n]\setminus S$ and then ``normalizing" the remaining points and hooks to get a valid hook configuration on a permutation in $S_{|S|}$. One can show that the map 
\begin{equation}\label{Eq1}\VHC(\Av_n(312))\to\bigcup_{r=0}^n\RedVHC(\Av_r(312))\times {[n]\choose r}
\end{equation} given by 
\[(\pi,V)\mapsto ((\pi,V)\big\vert_S,S)\] is a bijection. It follows that \[|\VHC(\Av_n(312))|=\sum_{r=0}^n{n\choose r}|\RedVHC(\Av_r(312))|.\] Combining this with Corollary \ref{cor:count with w}, one can show that \begin{equation}\label{Eq2}
|\RedVHC(\Av_n(312))|=\sum_{i=0}^n(-1)^i\mathfrak w(n-i-1),
\end{equation} where we make the convention $\mathfrak w(-1)=1$. 

For fixed $k\geq 1$, we can refine the map in \eqref{Eq1} to see that the sequence of numbers $|\VHC_k(\Av_n(312))|$ can be written as a nonnegative integer linear combination of ${n\choose 2k+1},\ldots,{n\choose 3k}$. Indeed, one can show that $\RedVHC_k(\Av_n(312))$ is only nonempty when $2k+1\leq n\leq 3k$ and that \[|\VHC_k(\Av_n(312))|=\sum_{i=1}^k{n\choose 2k+i}|\RedVHC_{k}(\Av_{2k+i}(312))|.\]
This motivates us to consider the numbers $|\RedVHC_{k}(\Av_{2k+i}(312))|$ for $1\leq i\leq k$. We can arrange these coefficients in a triangle as follows, where the $k^\text{th}$ row lists the coefficients of $|\VHC_{k}(\Av_{2k+i}(312))|$ as they range over $i$.
\[\begin{array}{rrrrrrr}
1
\\3 & 5
\\14 & 51 & 42
\\84 & 485 & 849 & 462
\\594 & 4743 & 13004 & 14819 & 6006
\\4719 & 48309 & 182311 & 322789 & 271452 & 87516
\\40898 & 511607 & 2472322 & 5999489 & 7794646 & 5182011 & 1385670
\end{array}\]

Invoking Corollary 5.1 from \cite{hannaspaper}, one can show that the first column of this table is the sequence of numbers $C_kC_{k+2}-C_{k+1}^2$ (OEIS sequence A005700), where $C_m=\frac{1}{m+1}{2m\choose m}$ denotes the $m^\text{th}$ Catalan number. Note that the row sums of this table are given by \eqref{Eq2}. The data in this table suggests some other remarkable patterns, which we state as conjectures. 

\begin{conjecture}\label{Conj1}
For every $k\geq 1$, we have \[|\RedVHC_k(\Av_{3k}(312))|=2\frac{(3k)!}{k!(k+1)!(k+2)!}.\]
\end{conjecture}   

The numbers on the right-hand side of the equation in Conjecture \ref{Conj1} are known as ``3-dimensional Catalan numbers." They are given in the OEIS sequence A005789. 

\begin{conjecture}\label{Conj2}
For every $k\geq 1$, we have \[\sum_{i=1}^k(-1)^{k-i}|\RedVHC_k(\Av_{2k+i}(312))|=C_k.\]
\end{conjecture}   

\begin{conjecture}\label{Conj3}
For every $k\geq 1$, the polynomial \[\sum_{i=1}^k|\RedVHC_k(\Av_{2k+i}(312))|x^{k-i}\] has only real roots. 
\end{conjecture}   

Note that Conjecture \ref{Conj3} implies the weaker statement that each row of the above table is log-concave. This in turn implies the even weaker statement that these rows are unimodal. Even if Conjecture \ref{Conj3} is out of reach, it would still be very interesting to prove the weaker statement that these rows are unimodal. 

\subsection{Relative Cardinalities of $\VHC(\Av_n(\tau))$}
Combining Theorem \ref{thm:312 asymtotics} with results from \cite{Defant}, we see that we have fairly precise asymptotic information about the sequence $|\VHC(\Av_n(\tau))|$ whenever $\tau\in S_3\setminus\{321\}$. However, the pattern $321$ remains mysterious in this context. As Defant remarks in \cite{Defant}, it would be interesting to prove anything nontrivial about the sets $\VHC(\Av_n(321))$. We can compute the sizes of these sets for $n\leq 13$ and see that they ostensibly grow much more quickly than $|\VHC(\Av_n(\tau))|$ for other length-$3$ patterns $\tau$. 

This leads us to speculate the following conjecture, which we know to be true for $n\leq r$. Recall that the \emph{weak Bruhat order} on $S_r$ is the partial order $\leq_B$ on $S_r$ generated by saying that $\sigma\leq_B\tau$ if there exists $i\in[r-1]$ such that $\sigma_i<\sigma_{i+1}$ and such that $\tau$ is obtained from $\sigma$ by swapping the entries $\sigma_i$ and $\sigma_{i+1}$. 

\begin{conjecture}\label{Conj4}
Choose $\sigma,\tau\in S_r$ such that $\sigma\leq_B\tau$. Then $|\VHC(\Av_n(\sigma))|\leq|\VHC(\Av_n(\tau))|$ for all $n\geq 1$. 
\end{conjecture}

Moreover, by Corollary \ref{cor:swr properties} the map $\W_n$ effectively changes every occurrence of the pattern 132 in a permutation to an occurrence of 312 and updates the valid hook configuration accordingly. This leads us to wonder if we could generalize $\W_n$ to an injection $\W_{n,\sigma,\tau}:\VHC(\Av_n(\sigma))\to\VHC(\Av_n(\tau))$ whenever $\sigma$ and $\tau$ differ by a single transposition. It is impossible for $\W_{n,\sigma,\tau}$ to always preserve the number of hooks like $\W_n$ does; a simple counterexample is that $\VHC_2(\Av_5(312))>\VHC_2(\Av_5(321))$. However, it would be interesting to ask if there is in general such a $\W_{n,\sigma,\tau}$ where $\W_{n,\sigma,\tau}(\pi, V)$ always has at most as many hooks as $(\pi, V)$. In the $r=3$ case, the existence of such a map is supported numerically. This question is also natural in that replacing occurrences of $\tau$ with occurrences of $\sigma$ in $\pi$ intuitively should not increase the number of descents.

\section{Acknowledgements}

This research was conducted at the REU at the University of Minnesota Duluth, which is supported by NSF/DMS grant 1659047 and NSA grant H98230-18-1-0. I would like to thank Joe Gallian for running the program and Amanda Burcroff and Avi Zeff for their proofreading suggestions. In addition, I am especially thankful to Colin Defant for suggesting the project and for invaluable advice throughout, including suggestions on how to prove Proposition \ref{prop:not D finite} and Theorem \ref{thm:312 asymtotics}.

\end{document}